\documentclass{amsart}

\usepackage[mathscr]{eucal}
\usepackage{amssymb}
\usepackage{pifont}
\usepackage{graphicx}
\usepackage{psfrag} 
\usepackage{epstopdf} 
\usepackage[usenames,dvipsnames]{color}
\usepackage[normalem]{ulem}
\usepackage{amsthm}
\usepackage{bbold}
\usepackage{bbm}
\usepackage{enumerate}
\usepackage{array}
\usepackage{amsmath}

\usepackage{hyperref}
\hypersetup{colorlinks=true,linkcolor={Brown},citecolor={Brown},urlcolor={Brown}}

\usepackage{cleveref}

\numberwithin{equation}{section}
\makeindex

\usepackage[all]{xy}
\CompileMatrices

\newdir{ >}{{}*!/-10pt/\dir{>}}

\swapnumbers 

\newtheorem{Thm}[equation]{Theorem}
\newtheorem{Prop}[equation]{Proposition}
\newtheorem{Lem}[equation]{Lemma}
\newtheorem{Cor}[equation]{Corollary}

\theoremstyle{remark}
\newtheorem{Rem}[equation]{Remark}
\newtheorem*{Rem*}{Remark}
\newtheorem{Def}[equation]{Definition}
\newtheorem{Ter}[equation]{Terminology}
\newtheorem{Rec}[equation]{Recollection}
\newtheorem{Not}[equation]{Notation}
\newtheorem{Exa}[equation]{Example}

\newtheorem{Cons}[equation]{Construction}

\theoremstyle{definition}
\newtheorem*{Ack*}{Acknowledgements}

\newcommand{\nc}{\newcommand}
\nc{\dmo}{\DeclareMathOperator}

\dmo{\Ab}{Ab}
\dmo{\AbMon}{AbMon}
\dmo{\Aut}{Aut}
\dmo{\bicMack}{\biMack_{\mathsf{ic}}} 
\dmo{\biMack}{\mathsf{Mack}} 
\dmo{\boxprod}{\Box}
\dmo{\Ch}{Ch}
\dmo{\CoInd}{CoInd}
\dmo{\Der}{D}
\dmo{\End}{End}
\dmo{\Free}{Free} 
\dmo{\Fun}{\mathrm{Fun}} 
\dmo{\Hom}{Hom}
\dmo{\Ho}{Ho}
\dmo{\img}{im}
\dmo{\Img}{Im}
\dmo{\incl}{incl}
\dmo{\Ind}{Ind}
\dmo{\Infl}{Inf}
\dmo{\Defl}{Def}
\dmo{\Iso}{Iso}
\dmo{\Inj}{Inj} 
\dmo{\Ker}{Ker}
\dmo{\Lan}{Lan} 
\dmo{\Ran}{Ran} 
\dmo{\Mackey}{Mack} 
\dmo{\coMackey}{coMack}
\dmo{\Map}{Map}%
\dmo{\Mod}{Mod}
\dmo{\Mor}{Mor}%
\dmo{\Obj}{Obj}
\dmo{\Perm}{Perm}
\dmo{\Proj}{Proj} 
\dmo{\pr}{pr}
\dmo{\PsFunJJ}{\PsFun_{\JJ_!}^{\JJ^\prime\textrm{\!-}\mathsf{oplax}}}
\dmo{\PsFunJop}{\PsFun_{{{\JJ}_{{}_{*}}}}}
\dmo{\PsFunJ}{\PsFun_{\JJ_!}}
\dmo{\PsFunoplax}{\PsFun^{\mathsf{oplax}}}
\dmo{\PsFun}{\mathsf{PsFun}} 
\dmo{\Res}{Res}
\dmo{\Rep}{Rep}
\dmo{\SH}{SH}
\dmo{\Spanname}{\mathsf{Span}}
\dmo{\Stab}{Stab}
\dmo{\twoFun}{2\mathsf{Fun}}

\dmo{\dom}{dom} 
\dmo{\cod}{cod} 
\dmo{\Yo}{Yo}

\nc\noloc{\nobreak\mspace{6mu plus 1mu}{:}\nonscript\mkern-\thinmuskip\mathpunct{}\mspace{2mu}}
\nc{\ababs}{{\sl ab absurdo}}
\nc{\Add}{\mathsf{Add}}
\nc{\ADD}{\mathsf{ADD}}
\nc{\adhoc}{{\sl ad hoc}}
\nc{\adjto}{\rightleftarrows}
\nc{\adj}{\dashv\,}
\nc{\afortiori}{{\sl a fortiori}}
\nc{\aka}{{a.\,k.\,a.}\ }
\nc{\all}{\mathsf{all}}
\nc{\apriori}{{\sl a priori}}
\nc{\ass}{\mathrm{ass}} 
\nc{\bbA}{\mathbb{A}}
\nc{\bbB}{\mathbb{B}}
\nc{\bbC}{\mathbb{C}}
\nc{\bbD}{\mathbb{D}}
\nc{\bbF}{\mathbb{F}}
\nc{\bbI}{\mathbb{I}}
\nc{\bbM}{\mathbb{M}}
\nc{\bbN}{\mathbb{N}}
\nc{\bbP}{\mathbb{P}}
\nc{\bbQ}{\mathbb{Q}}
\nc{\bbR}{\mathbb{R}}
\nc{\bbZ}{\mathbb{Z}}
\nc{\bs}{\backslash}
\nc{\BurnG}{\cat{A}(G)}
\nc{\cat}[1]{\mathcal{#1}}
\nc{\Cat}{\mathsf{Cat}}
\nc{\CAT}{\mathsf{CAT}}
\nc{\cf}{{\sl cf.}\ }
\nc{\Cf}{{\sl Cf.}\ }
\nc{\colim}{\mathop{\mathrm{colim}}}
\nc{\costar}{**}
\nc{\co}{{\mathrm{co}}}
\nc{\DD}{\cat{D}}
\nc{\Displ}{\displaystyle}
\nc{\doublequot}[3]{#1\backslash #2/#3}
\nc{\Ecell}{\rotatebox[origin=c]{90}{$\Downarrow$}} 
\nc{\eg}{{\sl e.g.}\ } 
\nc{\Eg}{{\sl E.g.}\ } 
\nc{\eps}{\varepsilon}
\nc{\equalby}[1]{\overset{\textrm{#1}}{=}}
\nc{\exact}{\mathsf{ex}}
\nc{\faithful}{\mathsf{faithful}}
\nc{\faith}{\mathsf{faithf}}
\nc{\final}{\textrm{\scriptsize{\ding{93}}}} 
\nc{\Funadd}{\Fun_{\amalg}}
\nc{\Funplus}{\Fun_{+}}
\nc{\fun}{\mathrm{fun}} 
\nc{\FPk}{\mathrm{FP}_\Bbbk} 
\nc{\GG}{\mathbb{G}}
\nc{\gpdG}{{\groupoidf_{\!\smallslash\!G}}} 
\nc{\gpdGfuz}{{\groupoid{}^{\smallfaithful,\mathsf{fus}}_{\!\smallslash\!G}}}
\nc{\ssetfuz}{\sset^{\smallfused}} 
\nc{\gpd}{\groupoid}
\nc{\GinG}{{\groupoidf_{G}}}
\nc{\gps}{\mathsf{groups}} 
\nc{\groconn}{\groupoid_{\mathsf{conn}}}
\nc{\groupoidf}{\groupoid{}^{\smallfaithful}}
\nc{\groupoid}{\mathsf{gpd}}
\nc{\group}{\mathsf{grp}} 
\nc{\Gsets}{G\mathsf{-sets}}
\nc{\HGfK}{\doublequot{H}{G}{f(K)}}
\nc{\HGK}{\doublequot HGK}
\nc{\Homcat}[1]{\Hom_{\cat #1}}
\nc{\hooklongleftarrow}{\longleftarrow\joinrel\rhook}
\nc{\hooklongrightarrow}{\lhook\joinrel\longrightarrow}
\nc{\hook}{\hookrightarrow}
\nc{\Hsets}{H\mathsf{-sets}}
\nc{\ICAdd}{\Add_{\mathsf{ic}}}
\nc{\ICADD}{\ADD_{\mathsf{ic}}}
\nc{\Idcat}[1]{\Id_{\cat{#1}}}
\nc{\id}{\mathrm{id}}
\nc{\Id}{\mathrm{Id}}
\nc{\ie}{{\sl i.e.}\ }
\nc{\into}{\mathop{\rightarrowtail}}
\nc{\inv}{^{-1}}
\nc{\Iout}[1]{\Ivo{\sout{#1}}}
\nc{\isocell}[1]{\undersett{ #1}{\overset{\sim}{\Ecell}}} 
\nc{\Isocell}[1]{\undersett{ #1}{\overset{\sim}{\Longrightarrow}}}
\nc{\isoEcell}{\overset{\sim}{\Rightarrow}} 
\nc{\Ivocell}[1]{\overset{\scriptstyle{#1}}{\Ecell}}  
\nc{\isotoo}{\stackrel{\sim}\longrightarrow}
\nc{\isoto}{\buildrel \sim\over\to}
\nc{\Ivo}[1]{{\color{OliveGreen}#1}} 
\nc{\James}[1]{{\color{Red}#1}} 
\nc{\JJ}{\mathbb{J}}
\nc{\kk}{\Bbbk}
\nc{\KK}{\mathrm{KK}}
\nc{\leps}{{}^{\ell}\eps}
\nc{\leta}{{}^{\ell}\eta}
\nc{\loccit}{{\sl loc.\ cit.}}
\nc{\lotoo}[1]{\overset{#1}{\,\longleftarrow\,}}
\nc{\loto}[1]{\overset{#1}{\leftarrow}}
\nc{\lto}{\leftarrow}
\nc{\lun}{\mathrm{lun}} 
\nc{\Mackintro}[1]{(Mack\,\ref{Mack-#1-intro})}
\nc{\Mack}[1]{(Mack\,\ref{Mack-#1})}
\nc{\Mid}{\,\big|\,}
\nc{\MMod}{\,\text{-}\!\Mod}%
\nc{\MM}{\cat{M}}
\nc{\Muniv}{\cat{M}_{\mathsf{univ}}}
\nc{\Ncell}{\rotatebox[origin=c]{0}{$\Uparrow$}} 
\nc{\NEcell}{\rotatebox[origin=c]{135}{$\Downarrow$}} 
\nc{\NN}{\cat{N}}
\nc{\NWcell}{\rotatebox[origin=c]{-135}{$\Downarrow$}} 
\nc{\oEcell}[1]{\overset{\scriptstyle #1}{\Ecell}} 
\nc{\oWcell}[1]{\overset{\scriptstyle #1}{\Wcell}} 
\nc{\ointo}[1]{\overset{#1}{\rightarrowtail}}
\nc{\olto}[1]{\overset{#1}\lto}
\nc{\onto}{\mathop{\twoheadrightarrow}}
\nc{\op}{{\mathrm{op}}}
\nc{\otoo}[1]{\overset{#1}{\,\longrightarrow\,}}
\nc{\oto}[1]{\overset{#1}\to}
\nc{\Paul}[1]{{\color{Blue}#1}}
\nc{\pih}[1]{\tau_{1}#1}
\nc{\Pout}[1]{\Paul{\sout{#1}}}
\nc{\PsFunJindex}{\PsFun_{{\JJ_!}} \ \ {{\JJ}_{!}}\textrm{-strong pseudo-functors}}
\nc{\qquadtext}[1]{\qquad\textrm{#1}\qquad}
\nc{\quadtext}[1]{\quad\textrm{#1}\quad}
\nc{\ra}{\rightarrow}
\nc{\reps}{{}^{r\!}\eps}
\nc{\restr}[1]{{|_{\scriptstyle #1}}}
\nc{\reta}{{}^{r\!}\eta}
\nc{\run}{\mathrm{run}} 
\nc{\Sad}{\mathsf{Sad}}
\nc{\SAD}{\mathsf{SAD}}
\nc{\sbull}{{\scriptscriptstyle\bullet}}
\nc{\Scell}{\rotatebox[origin=c]{0}{$\Downarrow$}} 
\nc{\SEcell}{\rotatebox[origin=c]{45}{$\Downarrow$}} 
\nc{\SET}[2]{\big\{\,#1\Mid#2\,\big\}}
\nc{\set}{\mathrm{set}} 
\nc{\Set}{\mathrm{Set}}
\nc{\smallfaithful}{\mathsf{f}}
\nc{\smallfused}{\mathsf{fus}}
\nc{\ssetfused}{\textrm{-}\underline{\set}} 
\nc{\smallslash}{{}^{\scriptscriptstyle/}}
\nc{\smat}[1]{\left(\begin{smallmatrix} #1 \end{smallmatrix}\right)}
\nc{\spanG}{{\widehat{\mathsf{gp}\,\,}\!\!\mathsf{d}}{}^\smallfaithful_{\!{}^{\scriptscriptstyle/}\!G}}
\nc{\Spanhat}{\textrm{\sf S}\widehat{\textrm{\sf pan}}} %
\nc{\Span}{\mathsf{Span}}
\nc{\spancat}{\mathrm{Sp}} 
\nc{\spank}{\spancat_{\kk}} 
\nc{\biset}{\mathsf{biset}}
\nc{\Biset}{\mathsf{Biset}}
\nc{\bisetcat}{\mathrm{Bis}} 
\nc{\bisk}{\bisetcat_\kk} 
\nc{\bifree}{\mathsf{bif}}
\nc{\rfree}{\mathsf{rf}}
\nc{\lfree}{\mathsf{lf}}
\nc{\sset}{\textrm{-}\set}
\nc{\str}{\mathsf{str}}
\nc{\SWcell}{\rotatebox[origin=c]{-45}{$\Downarrow$}} 
\nc{\too}{\mathop{\longrightarrow}\limits} 
\nc{\tSpan}{\pih{\Spanname}}
\nc{\undersett}[1]{\underset{\scriptstyle #1}}
\nc{\un}{\mathrm{un}} 
\nc{\unit}{\mathbb{1}}
\nc{\vcorrect}[1]{{\vphantom{\vbox to #1em{}}}}
\nc{\Wcell}{\rotatebox[origin=c]{90}{$\Uparrow$}} 
\nc{\what}[1]{\widehat{\cat{#1}}}
\nc{\xra}{\xrightarrow}

\begin{document}


\title{On the comparison of spans and bisets}
\author{Ivo Dell'Ambrogio}
\author{James Huglo}
\date{\today}

\address{
\noindent Univ.\ Lille, CNRS, UMR 8524 - Laboratoire Paul Painlev\'e, F-59000 Lille, France}
\email{ivo.dell-ambrogio@univ-lille.fr}
\email{james.huglo@univ-lille.fr}
\urladdr{http://math.univ-lille1.fr/$\sim$dellambr}

\begin{abstract} 
We compare the bicategory of spans with that of bisets (a.k.a.\  bimodules, distributors, profunctors) in the context of finite groupoids. We construct in particular a well-behaved pseudo-functor from spans to bisets. 
This yields an application to the axiomatic representation theory of finite groups, namely a new proof and a strengthening of Ganter and Nakaoka's identification of Bouc's category of biset functors as a reflective subcategory of global Mackey functors. To this end, we also prove a tensor-monadicity result for linear functor categories.  
\end{abstract}

\thanks{Authors partially supported by Project ANR ChroK (ANR-16-CE40-0003) and Labex CEMPI (ANR-11-LABX-0007-01).}


\maketitle


\tableofcontents

\section{Introduction and results}
\label{sec:intro}%

The motivation behind this paper is to provide a new proof of a result of Nakaoka \cite{Nakaoka16}~\cite{Nakaoka16a} identifying the tensor category of biset functors as a full tensor ideal subcategory of global Mackey functors (see \Cref{Cor:bisetfun} below, where we state the result in question after some recollections). 
In our approach, Nakaoka's theorem arises as a formal consequence of a `higher' result of independent interest, comparing two bicategories whose objects are, in both cases,  finite groupoids: 
\begin{enumerate} [\rm(1)]
\item The \emph{bicategory of spans}, denoted~$\Span$. In this bicategory, a 1-morphism $H\to G$ between two groupoids $H,G$ is a span of functors $H \gets S \to G$, and a 2-morphism is an equivalence class of diagrams
\[
\xymatrix@R=8pt@C=10pt{
&& S \ar[dd] \ar[dll] \ar@{}[ddl]|{\simeq \;\;} \ar@{}[ddr]|{\;\;\simeq} \ar[drr] && \\
H &&  && G \\
&& S' \ar[ull] \ar[urr] &&
}
\]
of functors and natural isomorphisms. Horizontal composition is computed by forming iso-comma squares. See details in \Cref{Cons:Span}.
\item The \emph{bicategory of bisets}, denoted~$\Biset$. Here a 1-morphism $H\to G$ is a finite $G,H$-biset (a.k.a.\ distributor, profunctor, bimodule, relator), \ie a finite-sets-valued functor $H^\op\times G\to \set$, and 2-morphisms are simply natural transformations of such functors. Horizontal composition is computed by taking tensor products (coends) of functors. See details in \Cref{Cons:Biset}.
\end{enumerate}

This is our comparison result, whose proof can be found in~\Cref{sec:R}:

\begin{Thm} [{Comparison of spans and bisets}]
\label{Thm:main1-intro}
There exists a pseudo-functor $\mathcal R \colon \Span \to \Biset$ we call \emph{realization}, which is the identity on objects (\ie finite groupoids) and `realizes' a span $H \stackrel{b}{\leftarrow} S \stackrel{a}{\to} G$ from $H$ to $G$ as  the coend
\[
\cat R (b,a) := G(a-,-) \otimes_S H(-,b-) := \int^{s\in S} \underbrace{G(as , -)}_{\mathcal R_!(a)} \times \underbrace{H( - , bs)}_{\mathcal R^*(b)} \,.
\]
Moreover, every biset is (canonically) isomorphic to the realization of a span.
\end{Thm}

\begin{Rem}
We are only interested in \emph{finite} groupoids, but everything can be easily extended to arbitrary ones.
The ingredients of \Cref{Thm:main1-intro} appear to be well-known to experts, such as the adjunctions $\cat R_!(u)\dashv \cat R^*(u)$ or the 
`moreover' part, and indeed the component functors $\mathcal R_{H,G}$ between Hom categories have been studied before in much detail; see \eg \cite{Benabou00pp}. 
A closely related statement appears, without proof, as Claim\,13 in \cite{Hoffnung12}. 
In fact a study of bisets (bimodules) between enriched categories in terms of (co)spans is already carried out in \cite{Street80}, using the machinery of `fibrations in bicategories'. Presumably, it should be possibly to unfold the layers of abstractions in \emph{loc.\,cit.} in order to derive \Cref{Thm:main1-intro} from the theory therein. 
Our proof, by constrast, strives to remain as concrete as possible.
\end{Rem}

For our application to biset and Mackey functors, we don't need the full bicategorical strength of \Cref{Thm:main1-intro} but rather only its 1-categorical, linearized shadow.
Fix a commutative ring~$\Bbbk$ and consider the $1$-truncated and $\Bbbk$-linearized versions
\[
\spank := \Bbbk \tau_1 (\Span  )
\quad\textrm{ and } \quad
\bisetcat_\kk := \Bbbk \tau_1 ( \Biset )
\]
of $\Span$ and~$\Biset$, obtained by identifying isomorphic 1-morphisms and by freely extending the resulting Hom abelian monoids to $\kk$-modules (\Cref{Rec:pih}-\ref{Rec:semi-add}).
By construction, $\spank$ and $\bisetcat_\kk$ are two $\kk$-linear additive categories. This means their Hom sets are $\kk$-modules (in fact free of finite type), their composition maps are $\kk$-bilinear, and they admit arbitrary finite direct sums induced by the disjoint unions of groupoids.

The category of $\kk$-linear representations (\ie the $\kk$-linear functor category)
\[
\cat M:= \Rep \spank  := \Fun_\kk (\spank , \kk\MMod)
\]
is, by definition, the category of \emph{global Mackey functors over~$\kk$}. 

\begin{Rem}
There are several versions of global Mackey functors; see \cite{DellAmbrogio19pp} for an overview. This one is defined for all finite groups and comes equipped with both inflation and deflation maps, besides induction, restriction and isomorphism maps. The present definition in terms of groupoids has appeared in \cite{Ganter13pp} and was reformulated in \cite{Nakaoka16} in terms of a certain 2-category~$\mathbb S$, which was later recognized in \cite{Nakaoka16a} to be biequivalent to the 2-category of groupoids.
\end{Rem}

Similarly, the representation category 
\[
\cat F:= \Rep \bisetcat_\kk :=\Fun_\kk (\bisetcat_\kk ,\kk \MMod) 
\]
is easily recognized to be Bouc's category of \emph{biset functors}~\cite{Bouc10} (see \Cref{Rem:bisetfun-comp}). Biset functors too can be understood as a variant of global Mackey functors, similarly defined on all finite groups and equipped with induction, restriction, inflation, deflation and isomorphism maps. 
Indeed, they can be shown to be equivalent to Webb's \emph{globally defined Mackey functors} \cite[\S8]{Webb00} for $\mathcal X$ and $\mathcal Y$ the class of all finite groups.
It is thus natural to compare the two notions, $\cat M$ and~$\cat F$.

After decategorifying and linearizing, \Cref{Thm:main1-intro} yields a full $\kk$-linear functor 
\[
F:=\kk\pih \cat R \colon \spank \longrightarrow \bisetcat_\kk
\]
which is the identity on objects.
In such a situation, it follows easily that precomposition with $F$ induces a fully faithful functor
\[
F^*\colon \cat F \hookrightarrow \cat M
\]
identifying $\cat F$ with a full reflexive $\kk$-linear subcategory of $\cat M$. This is precisely the embedding of Nakaoka's theorem we had mentioned at the beginning, and for which we have just given a transparent construction.

\begin{center} $***$ \end{center}

There is more to this story. Both global Mackey functors and biset functors form \emph{$\kk$-linear tensor categories}, by which we mean symmetric monoidal categories where the tensor functor $-\otimes -$ is $\kk$-linear in both variables (similarly below, by \emph{tensor functor} we will mean a strong symmetric monoidal $\kk$-linear functor.)
It is therefore natural to compare their tensor structures via the embedding~$F^*$. 

To this end, we first notice that both tensor products arise by Day convolution (\Cref{Cons:Day-convo}) from tensor structures on $\spank$ and~$\bisetcat_\kk$, both of which are induced by the cartesian product of groupoids. Moreover, $\spank$ and $\bisetcat_\kk$ are easily seen to be \emph{rigid}, in fact every object is its own tensor dual (\Cref{Ter:tensor}). 
In such a situation we can make use of the following general  abstract theorem:

\begin{Thm}
\label{Thm:main2-intro}
Let $F\colon \cat{C}\to \cat{D}$ be any $\kk$-linear tensor functor between two essentially small $\kk$-linear tensor categories.
Consider the diagram
$$\xymatrix{
& \Rep  \cat C 
\ar@<-.5ex>[dl]_{F_!} \ar@<-.5ex>[dr]_{\Free} &\\ 
\Rep \cat D  \ar[rr]_-E^-\sim 
\ar@<-.5ex>[ru]_{F^*}
&& A\MMod
 \ar@<-.5ex>[ul]_{U}
}$$
consisting of the following standard categorical constructions:
\begin{itemize}
\item
$\Rep  \cat C $ and $\Rep  \cat D $ are the $\kk$-linear categories of representations, as above, equipped with the respective Day convolution tensor products;
\item 
$F^*$ is the restriction functor along~$F$ and $F_!$ denotes its left adjoint, which is a tensor functor; it follows that $F^*$ is lax monoidal;
\item
 $A$ denotes the commutative monoid $F^*(\unit)$ in $\Rep \cat C$, whose multiplication map is induced by the lax monoidal structure of $F^*$ and the (unique) multiplication $\unit \otimes \unit \overset{\sim}{\to} \unit$ of the tensor unit object $\unit= \cat D(\unit_\cat D,-)$ of $\Rep \cat D $;
 \item
 $A\MMod$ denotes the category of left $A$-modules in $\Rep \cat C$, equipped with the tensor product $-\otimes_A-$ over~$A$; we also have the forgetful functor $U$ and its left adjoint $\Free$ sending a $\cat C$-representation $M$ to the free module~$A\otimes M$; 
 \item
 and where, finally, $E$ is the Eilenberg-Moore functor comparing the adjunction $F_!\dashv F^*$ with the adjunction $\Free\dashv U$, \ie $E$ is the unique functor such that $U \circ E = F^*$ and $E \circ F_! \simeq \Free$.
\end{itemize}
Then: 
\begin{enumerate}[\rm(1)]
\item
If $F$ is essentially surjective and the tensor categories~$\cat{C}$ and~$\cat D$ are rigid, $E$ is an equivalence of $\kk$-linear categories. 
\item
If moreover $F$ is full, $E$ is an equivalence of tensor categories. Also, the functors $F^*$ and $U$ are fully faithful and they identify their (equivalent) source tensor categories with the full essential image $\Img(F^*)=\Img(U)$ as a \emph{tensor ideal} in $\Rep \cat C$ \textup(meaning: if $M\in \Rep \cat C$ and $N\in \Img(F^*)$ then $M\otimes N \in \Img(F^*)$\textup).
\end{enumerate}
\end{Thm}

This theorem collects and improves a few more or less known categorical results. It may be understood as a `tensor monadicity' criterion for quotient functors of $\kk$-linear rigid tensor categories.
The proof can be found in \Cref{sec:monadicity}, together with quick recollections on all constructions involved. 
 As illustration, let us point out an easy special case:

\begin{Exa}
 If $C$ is a commutative $\kk$-algebra, it can be viewed as a rigid tensor $\kk$-linear category~$\cat C$ with a single object, whose endomorphism algebra is~$C$. 
 The multiplication of~$C$
also provides the tensor product of maps $a\otimes b:=ab$, which defines a functor $\otimes\colon \cat C\times \cat C\to \cat C$ by commutativity. If $f\colon C\to D$ is a surjective morphism of commutative $\kk$-algebras, it can be viewed as a functor $F\colon \cat C\to \cat D$ satisfying all the hypotheses of the theorem. In this case, $A$ is just the ring $D$ seen as a monoid object in the tensor category of $C$-modules. The theorem now simply says that the tensor category of $A$-modules $(M,\rho\colon A\otimes_C M\to M)$, in the abstract Eilenberg-Moore sense, inside the tensor category of $C$-modules, identifies with the tensor ideal subcategory of $D$-modules, in the usual sense. 
\end{Exa}

\begin{center} $***$ \end{center}

Now, in order to recover the results of \cite{Nakaoka16} and \cite{Nakaoka16a} we simply specialize \Cref{Thm:main2-intro} by taking $F\colon \cat C\to \cat D$ to be the $\kk$-linear tensor functor $\spank\to \bisetcat_\kk$  obtained from the realization pseudo-functor~$\cat R$ of \Cref{Thm:main1-intro}. 
We get:

\begin{Cor}[Biset functors vs global Mackey functors]
\label{Cor:bisetfun}
There is a canonical equivalence of tensor categories between:
\begin{enumerate}[\rm(1)]
\item The tensor category of biset functors $\mathcal F$ in the sense of Bouc~\cite{Bouc10}, that is, representations of the category of bisets equipped with Day convolution.
\item The category of global Mackey functors~$\cat M$ which are modules over the global Green functor~$A= \bisetcat_\kk(1,F -)$, obtained by restricting the Burnside biset functor along~$F\colon \spank\to \bisetcat_\kk$, equipped with the tensor product over~$A$.
\end{enumerate}
Moreover, both categories identify canonically with the reflexive full tensor ideal of $\cat M$ of those global Mackey functors $M$ satisfying the `deflative relation' 
\[
\mathrm{def}^G_{G/N} \circ \mathrm{inf}^G_{G/N}= \id_{M(G/N)}
\] 
for every normal subgroup $N $ of a group~$G$.
\end{Cor}

Details on the corollary's proof will be given in \Cref{sec:bisetfun}.

We would like to stress the similarity between the above corollary and the much older, and better known, results relating the category $\Mackey_\kk(G)$ of Mackey functors for a fixed finite group~$G$ and the category  $\coMackey_\kk(G)$ of \emph{cohomological} Mackey functors for~$G$. Indeed, this comparison can be obtained by the very same method, as follows.
Recall that by Lindner \cite{Lindner76} we may define
\[ \Mackey_\kk(G)= \Rep \spank (G) \]
where $\spank(G)$ is the $\kk$-linear category of finite left $G$-sets and isomorphism classes of spans of $G$-maps. Recall also that by Yoshida's theorem~\cite{Yoshida83b} we may define 
\[ \coMackey_\kk(G)= \Rep \Perm_\kk (G) \]
where $\Perm_\kk(G)$ is the category of finitely generated permutation $\kk G$-modules. 

Implicit in Yoshida's arguments, and made explicit \eg by Panchadcharam-Street \cite{PanchadcharamStreet07}, is the existence of a $\kk$-linear functor 
\[
\spancat_\kk(G) \longrightarrow \Perm_\kk(G)
\]
sending a left $G$-set $X$ to the permutation module~$\kk [X]$ and a span $X \overset{\;\alpha}{\gets} S\overset{\beta}{\to} Y$ of $G$-maps to $x \mapsto \sum_{s \in \alpha^{-1}(x)} \beta(s)$.
This is easily seen to be a tensor functor $F$ satisfying all the hypotheses of \Cref{Thm:main2-intro}. In this case the theorem yields:

\begin{Cor}[Cohomological vs ordinary Mackey functors]
\label{Cor:cohomack}
For every finite group~$G$, there is an equivalence of tensor categories between:
\begin{enumerate}[\rm(1)]
\item The category $\Rep \Perm_\kk(G)$ of representations of permutation $\Bbbk G$-modules.
\item The category of modules over the fixed-points Green functor $\FPk$ \textup(also known as $\mathrm H^0(-;\kk)$\textup) inside the tensor category $\Mackey_\kk(G)$ of Mackey functors for~$G$, equipped with the tensor product over~$\FPk$. 
\end{enumerate}
Moreover, both categories identify canonically with the reflexive full tensor ideal of those ordinary Mackey functors $M$ for~$G$ which satisfy the `cohomological relation'
\[
\mathrm{ind}_L^H \circ \mathrm{res}^H_L=[H:L] \cdot \id_{M(H)} 
\]
for all subgroups $L\leq H\leq G$. 
\end{Cor}

Details on this corollary's proof will be given in \Cref{sec:cohomack}.

\begin{Rem}
Essentially the same way of comparing the various descriptions of cohomological Mackey functors was already explained in~\cite[\S10]{PanchadcharamStreet07}. Our present exposition also makes explicit the identification of the associated tensor structures.
\end{Rem}

\begin{Rem}
This article is based on the second author's PhD thesis. 
Notations and conventions have been adapted in order to agree with those of the monograph \cite{BalmerDellAmbrogio20} and the survey article \cite{DellAmbrogio19pp}, where groupoids and spans are similarly used in order to compare various kinds of Mackey (1- and 2-)functors. 
\end{Rem}

\begin{Ack*}
We are grateful to Paul Balmer and an anonymous referee for useful comments on the manuscript and to Steve Lack for pointing out to us the relevance of~\cite{Street80}.
\end{Ack*}

\section{Preliminaries on coends and linear coends}
\label{sec:preliminaries}%

In this article we make extensive use of coends and their calculation rules, both in the basic set-theoretic setting and in the linear setting over a commutative ring. 
Coends vastly generalize the familiar tensor products of modules over rings to more general (possibly enriched) functors. 
Heuristically, they are the universal procedure for identifying a left and a right action of the same functorial variable.

\begin{Rec}[{Coends; \cite[IX.6]{MacLane98}}]
\label{Rec:coend}
Consider a functor $H\colon \cat C^\op\times \cat C\to \Set$ into sets for some (essentially small) category~$\cat C$. The \emph{coend of~$H$} is a set denoted 
\[ \int^{c\in C}H(c,c) \quad\textrm{ or just }\quad \int^{c}H\]
which comes equipped with canonical maps $H(x,x)\to \int^cH$ (for all $x\in \Obj \cat C$) forming a dinatural transformation and satisfying a suitable universal property among all such. 
For our purposes, it will suffice to know that it can be computed by the following coequalizer in~$\Set$:
\[
\int^{c\in \cat C} H(c,c) 
= \mathrm{coeq}\left( 
\xymatrix{ 
{\displaystyle{\coprod_{(\alpha\colon c'\to c) \in \Mor \cat C} H(c,c') }} \ar@<.5ex>[r]^-{H(\id,\alpha)} \ar@<-.5ex>[r]_-{H(\alpha, \id)} & {\displaystyle{ \coprod_{c\in \Obj \cat C}} H(c,c) }
}
\right)
\]
Thus an element of the coend is the equivalence class $[x]_c$ of some $x\in H(c,c)$ for some object $c\in \cat C$, for the equivalence relation generated by setting $[x]_c = [x']_{c'}$ whenever there is some morphism $\alpha \in \cat C( c', c)$ and some $y\in H(c,c')$ such that $H(\alpha,\id)(y)=x$ and $H(\id,\alpha)(y)=x'$. (Note that if $\cat C$ is a groupoid the latter condition directly yields an equivalence relation, without the need to generate one.)

The canonical maps are the evident ones coming with the equalizer. 
\end{Rec}

We will need the following well-known (and easily verified) formulas:

\begin{Lem}[{Fubini; \cite[IX.8]{MacLane98}}]
\label{Lem:Fubini}
The coend of $H \colon (\cat C_1 \times \cat C_2)^\op\times (\cat C_1 \times \cat C_2)\to \Set$
can be computed one variable at a time, in either order:
\[
\int^{(c_1,c_2)\in \cat C_1\times \cat C_2} H \simeq 
\int^{c_1\in \cat C_1} \int^{c_2\in \cat C_2} H \simeq 
\int^{c_2\in \cat C_2} \int^{c_1\in \cat C_1} H\,.
\]
The isomorphisms are the identity maps on representatives $x\in H(c_1,c_2,c_1,c_2)$.
\qed
\end{Lem}

\begin{Lem}[{co-Yoneda}]
\label{Lem:coYoneda}
For every functor $M\colon \cat C\to \Set$ and every object $x\in \cat C$, there is an isomorphism
\[
\int^{c\in \cat C} \cat C(c,x) \times M(c) \simeq M(x)
\]
natural in~$x$, given by evaluation $[\alpha, m]_c \mapsto M(\alpha)(m)$ and with inverse given by $m\mapsto[\id_x, m]_x$.
\qed
\end{Lem}
\begin{center} $*\;*\;*$ \end{center}

In \Cref{sec:R}, the above set-theoretical coends will be used to horizontally compose bisets. However, most of the time we will work linearly over some base commutative ring~$\kk$, and in particular we will need to use the $\kk$-enriched version of coends. This requires replacing the `base' Cartesian category of sets with the tensor category of $\kk$-modules.

Fix the commutative ring~$\kk$. 

\begin{Not} We denote by $\kk\MMod$ the category of all $\kk$-modules and $\kk$-linear maps. It is a complete and cocomplete abelian category. It is also a tensor category, \ie a symmetric monoidal category, by the usual tensor product $-\otimes_\kk-$ over~$\kk$.
\end{Not}

\begin{Ter}
A \emph{$\kk$-linear category} $\cat C$ is a category enriched over $\kk$-modules: its Hom sets $\cat C(x,y)$ are equipped with the structure of a $\kk$-module and the composition maps $\cat C(y,z) \times \cat C(x,y)\to \cat C(x,z)$ are $\kk$-bilinear. If $\cat C, \cat D$ are $\kk$-linear categories, a \emph{$\kk$-linear functor} $F\colon \cat C\to \cat D$ is a functor $F$ from $\cat C$ to $\cat D$ such that its component maps $F=F_{x,y}\colon \cat C(x,y)\to \cat D(x,y)$ are all $\kk$-linear.
\end{Ter}

\begin{Rec}[$\kk$-linear coends]
\label{Rec:k-coend}
By replacing $\Set$ with $\kk\MMod$ and requiring everything to be $\kk$-linear in \Cref{Rec:coend}, we get the notion of a \emph{$\kk$-linear} or \emph{$\kk$-enriched} coend. Thus, concretely, for a $\kk$-linear category $\cat C$ and a $\kk$-bilinear functor $H\colon \cat C^\op \times \cat C\to \kk\MMod$ (or equivalently a $\kk$-linear functor $H\colon \cat C^\op\otimes_\kk \cat C\to \kk\MMod$ for the appropriate notion of tensor product of $\kk$-categories), the $\kk$-linear coend of~$H$, again denoted $\int^c H$ or $\int^{c\in \cat C}H(c,c)$, is the $\kk$-module computed by the same coend diagram as in \Cref{Rec:coend} but now taken in $\kk\MMod$. Hence a general element of $\int^c H$ is now a \emph{finite $\kk$-linear combination} of classes $[x]_v$ (for $v\in \Obj \cat C$ and $x\in H(c,c)$).

We will occasionally refer to such simple elements $[x]_v\in \int^c H$ as \emph{generators}. 
If $H= F \otimes_\kk G$ is an object-wise tensor product of two (or more) functors, as will often be the case, we will write $[x,y]_{v}$ rather than the cumbersome $[x\otimes y]_v$.
\end{Rec}

\begin{Rem}
\label{Rem:Yoneda-Fubini}
The Fubini \Cref{Lem:Fubini} and the co-Yoneda \Cref{Lem:coYoneda} also hold for $\kk$-linear coends, as $\kk$-linear isomorphisms, with the same proofs.
As the latter formula uses the tensor structure of the base category, it must be adapted and now takes the form of an isomorphism
\begin{equation} \label{eq:kcoYoneda}
\int^{c\in \cat C} \cat C(c,-) \otimes_\kk M(c) \simeq M
\end{equation}
of $\kk$-linear functors $\cat C \longrightarrow \kk\MMod$. 
(This is the special case $F=\Id_\cat C$ of a $\kk$-linear left Kan extension as in~\Cref{Cons:Lan-Ran-Rep} below.)
\end{Rem}

\section{Tensor-monadicity for functor categories}
\label{sec:monadicity}%

This section is dedicated to the proof of \Cref{Thm:main2-intro}, and to recalling all the relevant categorical constructions.

Fix throughout a commutative ring~$\kk$. We will compute with $\kk$-linear coends, as in \Cref{Rec:k-coend}.

\begin{Not}
\label{Not:Funk}
Let $\cat C$ be a small $\kk$-linear category (one with only a set of objects), or more generally, an essentially small one (one equivalent to a small (sub-)category). Then we may consider its \emph{category of representations}, namely the category
\[
\Rep \cat C := \Fun_\kk (\cat C, \kk\MMod)
\]
of all $\kk$-linear functors into $\kk$-modules and natural transformations between them. 
Note that $\Rep \cat C$ is again a $\kk$-linear category and is abelian, complete and cocomplete; its $\kk$-action, limits and colimits are taken in $\kk\MMod$,  object-wise on~$\cat C$.
\end{Not}

\begin{Cons} [Kan extensions]
\label{Cons:Lan-Ran-Rep}
Let $F\colon \cat C\to \cat D$ be a $\kk$-linear functor between two essentially small $\kk$-linear categories. There is a \emph{restriction} functor $F^*\colon \Rep \cat D\to \Rep \cat C$ sending a $\kk$-linear functor $M\colon \cat D\to \kk\MMod$ to $M\circ F$. Evidently, $F^*$ is $\kk$-linear and exact. 
Since $\kk\MMod$ is complete and cocomplete, $F^*$ admits both a left and a right adjoint:
\[
\xymatrix{
\Rep \cat C \ar@/_3ex/[d]_{F_! \,:=\,\Lan_F} \ar@/^3ex/[d]^{\Ran_F \, =: \, F_*}  \\
\Rep \cat D \ar[u]|{F^*} 
}
\]
This is guaranteed by the theory of Kan extensions (\cite[X]{MacLane98}), which moreover provides explicit formulas for them.
Recall \eg that the left Kan extension~$F_!$ can be computed at each $M\in \Rep \cat C$ by the following ($\kk$-linear) coend:
\[
F_!(M) = \int^{x\in \cat D}  \cat D(Fx, -) \otimes_\kk M(x)
\quad \colon \quad \cat D \longrightarrow \kk\MMod
\]
(see \cite[X.7]{MacLane98} as well as \cite[(4.25)]{Kelly05} for the enriched version). 
\end{Cons}

\begin{Cons}[The Eilenberg-Moore adjunction]
\label{Cons:EM}
Recall (\eg from \cite[VI]{MacLane98}) that every adjunction $L\colon \cat A \rightleftarrows \cat B \,:\!R$ gives rise to a \emph{monad} $\mathbb A$ on the category~$\cat A$, that is a monoid $\mathbb A=(\mathbb A, \mu , \eta )$ in the endofunctor category $\End(\cat A)$. More precisely, as a functor we have $\mathbb A=R\circ L$; its multiplication is the natural transformation $\mu = R \varepsilon L \colon \mathbb A\circ \mathbb A = RLRL \Longrightarrow RL= \mathbb A$, where $\varepsilon\colon LR\Rightarrow \Id_\cat B$ is the counit of the adjunction; and its unit map is the unit of the adjunction, $\eta \colon \Id_\cat A \Rightarrow RL=\mathbb A$.

As with any such monad, we may define its \emph{Eilenberg-Moore category} $\mathbb A\MMod_{\cat A}$, whose objects are \emph{left modules} (a.k.a.\ algebras) in $\cat A$ over the monad. More precisely, an object of $\mathbb A\MMod_{\cat A}$ is a pair $(M, \rho)$ where $M\in \Obj \cat A$ and $\rho \colon \mathbb A M\to M$ is a map $\rho\colon RL(M)\to M$ in $\cat A$ satisfying the usual associativity and unit axioms of a left action, expressed by  commutative diagrams in~$\cat A$:
\[
\vcenter{
\xymatrix{ 
\mathbb A \mathbb A M \ar[d]_{\mathbb A \rho } \ar[r]^-{\mu_M} & \mathbb AM \ar[d]^{\rho} \\
\mathbb A M \ar[r]^-{\rho} & M 
}}
\quad\quad\quad\quad\quad
\vcenter{
\xymatrix{
M \ar@{=}[dr] \ar[r]^-{\eta_M}& \mathbb A M \ar[d]^{\rho} \\
& M
}
}
\]
A morphism $(M,\rho)\to (M',\rho')$ in $\mathbb A\MMod_\cat A$ is a morphism $\varphi\colon M\to M'$ preserving the actions:
\[
\vcenter{
\xymatrix{
\mathbb A M \ar[d]_{\rho} \ar[r]^-{\mathbb A \varphi} & \mathbb A M' \ar[d]^{\rho'} \\
M \ar[r]^-{\varphi} & M'
}
}
\]
There is an evident forgetful functor $U_\mathbb A\colon \mathbb A\MMod_\cat A\to \cat A$ which simply forgets the actions~$\rho$, as well as a left adjoint $F_\mathbb A$ sending any object $M\in \cat A$ to the \emph{free module} $(\mathbb AM, \mu_M \colon \mathbb A\mathbb AM \to \mathbb AM)$ and a morphism $\varphi$ to~$\mathbb A\varphi$. 
The adjunction $(F_{\mathbb A}, U_{\mathbb A})$ induces on $\cat A$ the same monad~$\mathbb A= RL = U_{\mathbb A}F_{\mathbb A}$. 
This is in fact the \emph{final} adjunction realizing~$\mathbb A$, in that there is a unique comparison functor $E=E_\mathbb A\colon \cat B\to \bbA\MMod_\cat A$ 
\[
\xymatrix{
& \cat A \ar@<-1ex>[dl]_L \ar[rd]_{F_\bbA\!\!} & \\
\cat B \ar[ru]_R \ar[rr]_-{E_\bbA} && {\bbA\MMod_{\cat A}} \ar@<-1ex>[ul]_{U_\bbA}
}
\]
such that $U_\bbA \circ E_\bbA = R$ and $F_\bbA = E_\bbA \circ L$. Concretely, $E$ sends an object $N\in \cat B$ to $E(N)= (RN, R\varepsilon_N\colon \bbA RN = RLR N\to RN)$ and a map $\varphi\colon N\to N'$ to~$R\varphi$. 

Note that if $L$ and $R$ are $\kk$-linear functors between $\kk$-linear categories, then $\bbA\MMod_\cat A$ is also a $\kk$-linear category and the forgetful, free module and comparison functors are all $\kk$-linear.
\end{Cons}

\begin{Prop} \label{Prop:plain-monadicity}
Let $F\colon \cat C\to \cat D$ be a $\kk$-liner functor between two essentially small $\kk$-linear categories. Suppose that $F$ is essentially surjective. Then the adjunction $F_!\colon \Rep \cat C \rightleftarrows \Rep \cat D\,: \! F^*$ of \Cref{Cons:Lan-Ran-Rep} is monadic, that is the comparison $E\colon  \Rep \cat D \overset{\sim}{\to} (F^*F_!)\MMod_{\Rep \cat C}$ is a (\,$\kk$-linear) equivalence.
\end{Prop}

\begin{proof}
This is a consequence of the Beck monadicity theorem \cite[VI.7]{MacLane98}. 
In fact, $F^*$ is an exact functor between two abelian categories and admits a left adjoint. In this situation, the hypotheses of Beck monadicity reduce easily to $F^*$ being faithful  (see  \eg \cite[Thm.\,2.1]{ChenChenZhou15}), and the latter follows from the essential surjectivity of~$F$. Indeed, if $\varphi\colon N\Rightarrow N'$ is a natural transformation such that $F^* \varphi =0 $, then we have $\varphi_{Fc}=0\colon NFc\to N'Fc$ for every $c\in \cat C$ and therefore also $\varphi_d=0\colon Nd\to N'd$ for all $d\in \cat D$, by way of some isomorphism $Fc\simeq d$ and the naturality of~$\varphi$.
\end{proof}

\begin{Ter}
\label{Ter:tensor}
By a \emph{tensor category} we always mean a symmetric monoidal category (\cite[XI]{MacLane98}).
We will write $\otimes$ and $\unit$ (possibly with some decoration) for the tensor functor and the tensor unit object.
A \emph{$\kk$-linear tensor category} is a category which is simultaneously a tensor category and a $\kk$-linear category and whose tensor functor $-\otimes -$ is $\kk$-linear in both variables. 
A tensor category $\cat C$ is \emph{rigid} if every object $X$ admits a \emph{tensor dual}~$X^\vee$, meaning that there is a ($\kk$-linear) isomorphism
$ \cat C(X \otimes Y,Z) \simeq \cat C(Y, X^\vee \otimes Z) $ natural in $Y,Z\in \cat C$; in other words, the endofunctors $X\otimes -$ and $X^\vee \otimes-$ on $\cat C$ are adjoint. If $\cat C$ is rigid, then $X\mapsto X^\vee$ extends canonically to an equivalence $\cat C\simeq \cat C^\op$ of $\kk$-linear tensor categories.
\end{Ter}

\begin{Cons} [The tensor category $A\MMod_\cat A$]
\label{Cons:tensor-EM}
Let $A=(A,m,u)$ be a monoid in a tensor category~$\cat A$; thus we have a multiplication $m \colon A\otimes A\to A $ and unit $ u\colon \unit \to A$ in $\cat C$ making the usual associativity and unit diagrams commute. Then $\bbA := A\otimes (-)\colon \cat C\to \cat C$ is a monad on~$\cat C$ with multiplication $\mu\otimes-$ and unit~$\eta\otimes-$, and we may form the Eilenberg-Moore  module category $A\MMod_\cat A:= \bbA\MMod_\cat A$ and the adjunction $F_A \dashv U_A$ of \Cref{Cons:EM}.
If $A$ is \emph{commutative} (meaning of course that $m = m \sigma$ where $\sigma \colon A\otimes A \simeq A \otimes A$ is the symmetry isomorphism of~$\cat C$) and $\cat A$ admits sufficiently many coequalizers,
then $A\MMod_\cat A$ inherits the structure of a tensor category. Its tensor unit is the left $A$-module~$A$,
its tensor functor $-\otimes_A-$ is defined for all $(M,\rho),(M',\rho') \in A\MMod_\cat A$ by the coequalizer
\begin{equation} \label{eq:equalizer-over-A}
\xymatrix{
M \otimes A \otimes M' \ar@<.5ex>[rr]^-{\id \otimes \rho} \ar@<-.5ex>[rr]_-{\rho \sigma \otimes \id }&&
 M \otimes M'\ar[r] & M \otimes_A M'
}
\end{equation}
equipped with the evident induced left $A$-action. 
The unit, associativity and symmetry isomorphisms for $A\MMod_\cat A$ are induced by those of~$\cat A$.  

Note that if the tensor category $\cat A$ is $\kk$-linear then evidently so is $A\MMod_\cat A$.
\end{Cons}

\begin{Cons}[The Day convolution product; \cite{Day70}]
\label{Cons:Day-convo}
Let $\cat C$ be an essentially small $\kk$-linear tensor category. 
The representation category $\Rep \cat C$ inherits from $\cat C$ a $\kk$-linear tensor structure, called \emph{Day convolution}, which can be characterized as the unique (closed $\kk$-linear) tensor structure $-\otimes_\cat C -$ on $\Rep \cat C$ which preserves colimits in both variables and which makes the Yoneda embedding $\cat C^\op\to \Rep \cat C$ a tensor functor. The Day convolution of $M,N\in \Rep \cat C$ is computed by the ($\kk$-linear) coend
\[
M \otimes_\cat C N = \int^{u,v \in \cat C} \cat C(u\otimes v, -) \otimes_\kk M(u) \otimes_\kk N(v) \,.
\]
The unit object in $\Rep \cat C$ is the functor $\unit:=\cat C(\unit,-)$ corepresented by the unit object of~$\cat C$.
The left unitor, right unitor, associator and symmetry of the Day convolution product are obtained by combining, in the evident way, those of $\cat C$ with some canonical identifications of coends (see \eg \cite[Prop.\,1.2.16]{Huglo19pp} for details).
\end{Cons}

\begin{Lem} 
\label{Lem:rigid-Day}
If the tensor category $\cat C$ is rigid, the Day convolution product of $\Rep \cat C$ can be computed by either one of the two coends
\begin{align*}
M\otimes_\cat C N 
 & \simeq \int^{v\in \cat C} M(v^\vee \otimes - ) \otimes_\kk N(v) \\
 & \simeq \int^{u \in \cat C} M(u) \otimes_\kk N(u^\vee \otimes -)
\end{align*}
where $x^\vee$ denotes the tensor-dual of an object $x\in \cat C$.
\end{Lem}

\begin{proof}
At each $c\in \cat C$, we define the first isomorphism to be the following composite:
\begin{align*}
(M\otimes_\cat C N)(c) & = \int^{u,v\in \cat C} \cat C(u\otimes v, c) \otimes_\kk M(u) \otimes_\kk N(v) && \eqref{Cons:Day-convo}\\
 & \simeq \int^{u,v\in \cat C} \cat C(u,v^\vee \otimes c) \otimes_\kk M(u) \otimes_\kk N(v) && \cat C \textrm{ rigid}\\
  & \simeq \int^{v\in \cat C} \int^{u\in \cat C} \cat C(u,v^\vee \otimes c) \otimes_\kk M(u) \otimes_\kk N(v) && \textrm{Fubini }\ref{Lem:Fubini}\\
   & \simeq \int^{v\in \cat C} \left( \int^{u\in \cat C} \cat C(u,v^\vee \otimes c) \otimes_\kk M(u) \right) \otimes_\kk N(v)   &&  && \\
    & \simeq \int^{v\in \cat C} M(v^\vee \otimes c) \otimes_\kk N(v) && \textrm{co-Yoneda }\ref{eq:kcoYoneda}
\end{align*}
The second-to-last isomorphism uses that $\otimes_\kk$ preserves colimits of $\kk$-modules in both variables. 
The second formula is proved similarly and will not be needed. 

If we trace a generator 
$[f \colon u\otimes v \to c, m, n]_{u,v} \in (M\otimes_\cat C N)(c)$ all the way, we see that it corresponds to $[M(\tilde f)(m),n]_v$
where $\tilde f$ is the map 
\[ 
\xymatrix@1{
 u \simeq u \otimes \unit \ar[r] & u \otimes v \otimes v^\vee \ar[r]^-{f \otimes \id} & c \otimes v^\vee \simeq v^\vee \otimes c
 }
\]
 which also uses the symmetry of~$\cat C$.
\end{proof}

\begin{Cons}[Lax right adjoints and projection map]
\label{Cons:lax-proj}
Let us again consider a general adjunction $L \colon \cat A \rightleftarrows \cat B \,:\! R$ with unit $\eta$ and counit~$\varepsilon$. 
Suppose that $\cat A$ and $\cat B$ are tensor categories and that the left adjoint $L$ is a tensor functor. The right adjoint $R$ inherits from $L$ the structure of a \emph{lax} tensor functor, that is, an (`external') multiplication~$\lambda_{Y,Y'}$
\[
\xymatrix@1{
R(Y) \otimes R(Y') \ar@{-->}[r]^-{\lambda_{Y,Y'}} \ar[d]_-{\eta} & R(Y \otimes Y') \\
RL(R(Y) \otimes R(Y')) \ar[r]^-\sim & R(LR(Y) \otimes LR (Y')) \ar[u]_-{R(\varepsilon \otimes \varepsilon)} 
}
\]
(for all $Y,Y'\in \cat B$) and a unit 
\[
\iota \colon 
\xymatrix{ \unit_\cat A \ar[r]^-{\eta} & RL(\unit_\cat A) \ar[r]^-{\sim} & R(\unit_\cat B)  }
\]
satisfying the same coherence constraints as for a tensor functor (\ie the `strong' case, when $\lambda$ and $\iota$ are invertible).

As all lax tensor functors, $R$ preserves monoids: If $Y=(Y,m,u)$ is a monoid in~$\cat B$, then $R(Y)$ inherits a monoid structure with multiplication and unit
\[
\xymatrix{ RY \otimes RY \ar[r]^-{\lambda_{Y,Y}} & R(Y \otimes Y) \ar[r]^-{Rm} & RY }
\quad \textrm{ and } \quad
\xymatrix{ \unit \ar[r]^-{\iota} & R\unit \ar[r]^-{Ru} & RY } .
\]
By applying this to the unique monoid structure $(\unit, m\colon \unit \otimes \unit \overset{\sim}{\to} \unit, \id_\unit)$ on the tensor unit of~$\cat B$, we obtain a distinguished commutative monoid $A:= R(\unit)$ in~$\cat A$.

The lax structure on $R$ also produces the \emph{projection map}
\begin{equation} \label{eq:proj-map}
\pi_{Y,X} \colon
\xymatrix@1{
R(Y) \otimes X \ar[r]^-{\id\otimes \eta} & R(Y) \otimes RL(X) \ar[r]^-{\lambda_{Y, LX}} & R( Y \otimes L(X))
}
\end{equation}
(for all $X\in \cat A$ and $Y\in \cat B$) which in many contexts, but not always, is an isomorphism called \emph{projection formula}.
\end{Cons}

\begin{Lem}[{\cite[Lem.\,2.8]{BalmerDellAmbrogioSanders15}}]
\label{Lem:monads-morph}
In the situation of \Cref{Cons:lax-proj}, the map 
\[
\pi \colon 
\xymatrix{ 
A \otimes X \ar[r]^-{\pi_{\unit,X} } &
  R(L \unit \otimes L X ) \simeq R (\unit \otimes L X ) \simeq R L (X) 
}
\quad\quad (\textrm{for }X\in \cat A) 
\]
is always a morphism of monads on~$\cat A$, between the monad obtained from the monoid $A=R(\unit_\cat B)$ and the monad $\bbA = RL$ obtained from the adjunction $L\dashv R$.
\qed
\end{Lem}

\begin{Exa}
\label{Exa:all-together}
Consider a $\kk$-linear functor $F\colon \cat C\to \cat D$ and the induced adjunction $F_! \dashv F^*$  as in \Cref{Cons:Lan-Ran-Rep}.
Suppose now that $\cat C, \cat D$ are $\kk$-linear tensor categories and that $F$ is a tensor functor. It is a well-known general fact that the left adjoint $F_!\colon \Rep \cat C\to \Rep \cat D$ is naturally a tensor functor with respect to the Day convolution products (see \eg \cite[Prop.\,1.3.5]{Huglo19pp}). 
Everything in \Cref{Cons:lax-proj} can be applied to $L:=F_!\dashv F^*=:R$. In particular $F^*$ is a lax tensor functor. 
For future reference,  its structure maps are given at each $c\in \cat C$ by
\[\iota_c \colon 
\xymatrix@1{
(\unit_{\Rep \cat C})(c) 
= \cat C(\unit, c)
\ar[r]^-F &
\cat D (F \unit , Fc) 
\simeq
\cat D(\unit , Fc)
= (F^* \unit_{\Rep \cat D} )(c)
}
\]
and (in the rigid case, the only one we will need) by $[n,n']_u \mapsto [n,n']_{v=Fu}$
\begin{align} \label{eq:laxF*}
\lambda_{N,N',c}\colon \big( F^* N \otimes_\cat C F^* N'  \big) (c)
&\overset{\eqref{Lem:rigid-Day}}{=} \int^{u\in \cat C} N(Fu^\vee \otimes Fc)) \otimes_\kk N'F(u) \\
&\longrightarrow
\int^{v\in \cat D} N(v^\vee \otimes Fc) \otimes_\kk N'(v) 
\overset{\eqref{Lem:rigid-Day}}{=} F^* (N \otimes_\cat D N') (c) \nonumber
\end{align}
for all $N,N'\in \Rep \cat D$ (see \cite[Cor.\,1.3.6]{Huglo19pp}).
\end{Exa}

\begin{Prop}[Projection formula] \label{Prop:proj-formula}
Let $F\colon \cat C\to \cat D$ be a $\kk$-linear tensor functor between essentially small rigid $\kk$-linear tensor categories $\cat C $ and~$\cat D$. 
Consider the induced adjunction $L=F_! \colon \Rep \cat C \rightleftarrows \Rep \cat D \,:\!  F^*=R$ with its tensor structure as in \Cref{Exa:all-together}. Then the projection formula holds, that is the canonical map
\[
\pi_{N,M}\colon F^*(N) \otimes_\cat C M \overset{\sim}{\longrightarrow} F^* \big( N \otimes_\cat D F_! (M) \big)
\]
of~\eqref{eq:proj-map} is an isomorphism for all $M\in \Rep \cat C$ and $N\in \Rep \cat D$.
\end{Prop}

\begin{proof}
For every object $c\in \cat C$, we compute as follows (explanations below):
\begin{align*}
(F^*N \otimes_\cat C M)(c)  
& \;=\; \int^{v\in \cat C} N(F(v^\vee \otimes c)) \otimes M(v)&& \text{by \eqref{Lem:rigid-Day}} \\
    & \;\simeq\; \int^{v\in \cat C} N(Fv^\vee \otimes Fc) \otimes M(v) &&  \\
      & \;\simeq\; \int^{v\in \cat C}  \left( \int^{d\in \cat D} \cat D(d^\vee, Fv^\vee) \otimes  N(d^\vee \otimes Fc) \right) \otimes M(v) && \textrm{co-Yoneda}\\
        & \;\simeq\; \int^{v\in \cat C}  \left( \int^{d \in \cat D} \cat D(Fv, d) \otimes  N(d^\vee \otimes Fc) \right) \otimes M(v) &&  \\
          & \;\simeq\; \int^{d\in \cat D} N(d^\vee \otimes Fc) \otimes  \left( \int^{v\in \cat C} \cat D(Fv, d) \otimes M(v) \right) && \textrm{Fubini} \\
            & \;= \; F^* \big( N \otimes_\cat D F_! M \big)(c) && \text{by \eqref{Lem:rigid-Day}} 
\end{align*}
This successively uses: 
\Cref{Lem:rigid-Day} for the Day convolution over~$\cat C$; 
the tensor structure of~$F$; 
the co-Yoneda isomorphism \eqref{eq:kcoYoneda} (applied to the functor $N((-)^\vee \otimes Fc)\colon \cat D^\op \simeq \cat D \to \kk\MMod$ to compute the value $N(Fv^\vee \otimes Fc)$, exploiting the fact that $d\mapsto d^\vee$ is a self-inverse equivalence $\cat D^\op \simeq \cat D$);  
again the equivalence $(-)^\vee\colon \cat D^\op\simeq \cat D$; 
the Fubini \Cref{Lem:Fubini} to exchange the two coends; and finally, the definition of $F_!$ and \Cref{Lem:rigid-Day} for the Day convolution over~$\cat D$. 

If we trace the fate of a generator $[n,m]_v \in (F^*N \otimes M)(c)$ all the way, for any $n\in N(F(v^\vee \otimes c))$, $m\in M(v)$ and $v\in \cat C$, we see that it maps to 
$[\tilde n , [\id_{Fv}, m]_v ]_{d=Fv}$, where $\tilde n\in N(Fv^\vee \otimes Fv)$ is the element matching~$n$ under $Fv^\vee \otimes Fc \simeq F(v^\vee \otimes c)$.
This is easily seen to agree with the value of $[n,m]_v$ under the canonical map~$\pi_{N,M}$, again computed by a direct inspection of the definitions. 
Thus $\pi_{N,M}$ is invertible.
\end{proof}

\begin{Cor} \label{Cor:mod-vs-mod}
If the tensor categories $\cat C$ and $\cat D$ are rigid, the canonical morphism between monads on $\Rep \cat C$
\[ \pi\colon A\otimes(-) := F^*(\unit_{\Rep \cat D})\otimes_\cat C (-) \;\overset{\sim}{\longrightarrow}\; F^*F_! =: \bbA
\] 
 is an isomorphism. In particular, it induces by precomposition an isomorphism 
\[
\pi^*\colon \bbA \MMod_{\Rep \cat C} \overset{\sim}{\longrightarrow} A\MMod_{\Rep \cat C}
\]
of their module categories, identifying the two Eilenberg-Moore adjunctions.
\end{Cor}

\begin{proof}
Immediate from \Cref{Lem:monads-morph} and \Cref{Prop:proj-formula}. 
Concretely, the isomorphism $\pi^*$ sends an $\bbA$-module $(M,\rho)$ to the $A$-module $(M,\rho \pi)$.
\end{proof}

\begin{Rem}
The special case of the projection formula used in \Cref{Cor:mod-vs-mod} is actually easy to see directly, since the relevant map~$\pi$
\[
A \otimes_\cat C M = \int^{c\in \cat C} \cat D(\unit , F (c^\vee \otimes - )) \otimes M(c) 
\overset{\sim}{\longrightarrow}
\int^{c\in \cat C} \cat D(Fc, F-) \otimes M(c)
= F^* F_! M
\]
is just  induced by $\cat D(\unit , F (c^\vee \otimes - ))\simeq \cat D(\unit , Fc^\vee \otimes F- ) \simeq \cat D(Fc, F-)$, the isomorphisms given by the tensor structure of $F$ and the tensor duality of~$\cat D$.
\end{Rem}

\begin{Lem} 
\label{Lem:fullF}
If $F\colon \cat C \to \cat D$ is essentially surjective and full, then $F^*$ is a fully faithful embedding $\Rep \cat D\hookrightarrow \Rep \cat C$.
An $M\colon \cat C\to \kk\MMod$ belongs to the (essential) image of $F^*$ if and only if it factors (up to isomorphism) through~$F$, uniquely if so.
\end{Lem}

\begin{proof}
This is well-known, see \eg \cite[Prop.\,1.3.2]{Huglo19pp} for a detailed proof.
\end{proof}

\begin{proof}[Proof of \Cref{Thm:main2-intro}]
We finally have at our disposal all the ingredients of the theorem and its proof. Let $F\colon \cat C\to \cat D$ be a $\kk$-linear tensor functor between essentially small $\kk$-linear rigid tensor categories, and consider all the constructions as listed in the theorem and recalled above. (Since $F$ is essentially surjective, if $\cat C$ is rigid then $\cat D$ is also automatically rigid, see \eg \cite[Prop.\,1.1.10]{Huglo19pp} for a full proof.)

Consider the composite functor
\begin{equation} 
\label{eq:E_A}
E_A\colon \xymatrix{
\Rep \cat D  \ar[r]^-{E_\bbA} & {\bbA \MMod_{\Rep \cat C}} \ar[r]^-{\pi^*}_-{\sim} & A \MMod_{\Rep \cat C}
}
\end{equation}
of the Eilenberg-Moore comparison functor $E_\bbA$ and the identification $\pi^*$ of the categories of $A$-modules with that of $\bbA$-modules, as in \Cref{Cor:mod-vs-mod} (this uses rigidity).
It sends $N\in \Rep \cat D$ to $F^*N\in \Rep \cat C$ equipped with the $A$-action 
\begin{equation*} 
\rho:= F^*(\varepsilon)\circ \pi \colon A \otimes_\cat C F^*N \simeq F^* F_! F^*N \longrightarrow F^*N .
\end{equation*}

If $F$ is essentially surjective, $E_\bbA$ is an equivalence by \Cref{Prop:plain-monadicity} and therefore so is~$E_A$; this proves part~(1) of the theorem. 
If moreover $F$ is full, the functors $F^*$ and $U$ are fully faithful by \Cref{Lem:fullF}; this proves a third of~(2).

Under all these hypotheses, the remaining claims of part~(2) on the tensor structures follow from \Cref{Prop:E-tensor-compat} below. 
This ends the proof of the theorem.
\end{proof}

\begin{Prop} \label{Prop:E-tensor-compat}
If $F\colon \cat C\to \cat D$ is full and essentially surjective and $\cat C$ and $\cat D$ are rigid, then the equivalence $E_A$ of \eqref{eq:E_A} is a (strong) tensor functor, and the embeddings $F^*$ and $U$ identify these tensor categories with  the full tensor ideal subcategory $\Img(F^*)= \Img(U)$ of~$\Rep \cat C$.
\end{Prop}

\begin{proof}
For $N,N'\in \cat D$, let $E_A(N)=(F^*N, \rho)$ and $E_A(N')=(F^*N',\rho')$ denote the two images in $A\MMod$.
Recall, from \Cref{Cons:tensor-EM} and \Cref{Exa:all-together}, the coequalizer defining $-\otimes_A-$ and the lax multiplication $\lambda_{N,N'}$ of~$F^*$:
\[
\xymatrix{
F^*N \otimes_\cat C A \otimes_\cat C F^*N' \ar@<.5ex>[rr]^-{\id\otimes \rho'} \ar@<-.5ex>[rr]_-{\rho\sigma \otimes \id} &&
 F^*N \otimes_\cat C  F^*N' \ar[d]_-{\lambda_{N,N'}}  \ar[rr]^-\omega &&
  F^*N \otimes_A F^*N' \ar@{-->}[dll]^{\;\;\quad\quad\lambda_{N,N'} \omega^{-1} \;=:\; \varphi_{N,N'}}
  \\
 && F^*(N \otimes_\cat D N')  && 
}
\]
We claim that, under the hypotheses, $\lambda_{N,N'}$ is invertible and so is the canonical projection~$\omega$ to the coequalizer. In particular, we obtain the dotted isomorphisms~$\varphi_{N,N'}$. 
Indeed, writing as in~\eqref{eq:laxF*} (thanks to rigidity), $\lambda_{N,N'}$ sends $[n,n']_u$ to $[n,n']_{Fu}$, and the inverse map sends $[n,n']_v$ to $[n,n']_u$, for any choice of $u\in \cat C$ with $Fu\simeq v$. 
To see why the latter works, assume for simplicity that $F$ is surjective on objects (\ie replace $\cat D$ with the equivalent strict image of~$F$). 
Now choose $u_v\in F^{-1}(v)$ for each~$v\in \cat D$. 
The resulting map $(n,n')_v\mapsto (n,n')_{u_v}$, call it~$\theta$, is well-defined on the classes $[n,n']_v$; indeed, every map $\psi\in \cat D(v, \overline{v})$ testifying of an `elementary' relation $(n,n')_v \sim (\overline{n},\overline{n}')_{\overline{v}}$ lifts to some $\varphi\in \cat C(u_v, u_{\overline{v}})$  by the fullness of~$F$, showing $[n,n']_{u_v}= [\overline{n},\overline{n}']_{u_{\overline{v}}}$, and we may lift any zig-zag of such maps by the surjectivity of $F$ on objects. Clearly $\lambda_{N,N'}\circ \theta=\id$.
Moreover for every~$u \in \cat C$ we have $F(u) = F(u_{Fu})$, hence by the fullness of~$F$ we may lift $\id_{Fu}$ to some map $u\to u_{Fu}$ in~$\cat C$, which implies that $\theta \circ \lambda_{N,N'}=\id$ as well. Thus $\lambda_{N,N'}$ is invertible (and $\theta$ does not depend on the choices).

As for~$\omega$, it is a general fact, neither requiring rigidity nor the hypotheses on~$F$, that $\lambda$ factors through it (see \cite[Lemma 1.4.2]{Huglo19pp}); as $\lambda$ is  invertible and $\omega$ is always an epimorphism, we conclude that $\omega$ is also invertible. Alternatively, the latter can also be checked directly in the rigid case.

Note that this does \emph{not} mean that $F^*$ is a strong tensor functor, because the unit map 
$F\colon \cat C(\unit, -) \to \cat D (F \unit , F-)  \simeq \cat D(\unit , F-)$ 
is still not necessarily invertible. 
Still, the identity map $F^*(\unit_{\Rep \cat D}) \to U \circ E_A(\unit_{\Rep \cat D})$ \emph{is} invertible, and is also the unique $A$-linear morphism $A=F^*(\unit_{\Rep \cat D}) \to E_A(\unit_{\Rep \cat D})$ extending along the unit map $\unit_{\Rep \cat C}\to A$ of the monoid~$A$, as one checks immediately.
Moreover, the morphisms $\varphi_{N,N'}$ are automatically $A$-equivariant, since the fullness and essential surjectivity of $F$ imply that every $M\in \Rep \cat C$ can have at most a unique $A$-module structure (\cite[Cor.\,1.3.11]{Huglo19pp}).

Altogether, we see that the maps $\varphi_{N,N'}$ and the identity $A \to E_A(\unit)$ belong to $A\MMod_{\Rep \cat C}$ and equip $E_A$ with a strong tensor structure; the commutativity of the coherence diagrams follows from that for the lax structure of~$F^*$.

Finally, let us verify that $\Img(F^*)=\{M\in \Rep \cat C\mid \exists N \textrm{ s.t.\ } M\simeq F^*N\}$ is a tensor ideal. Notice that for all $M\in \Rep \cat C$
\begin{align*}
M\in \Img(F^*) &\;\Longleftrightarrow\;  \textrm{the unit }\eta\colon M\to F^*F_!M \textrm{ is invertible}\\
&\;\Longleftrightarrow\;  \textrm{the unit }M\simeq \unit \otimes M \to A \otimes M \textrm{ is invertible},
\end{align*}
the first equivalence because $F^*$ is the inclusion of a full reflexive subcategory, the second because the equivalence $E_A$ matches the two adjunctions.
We deduce for all $M\in \Img(F^*)$ and $N\in \Rep \cat C$ that $M\otimes N \simeq (A\otimes M)\otimes N\simeq A \otimes (M\otimes N)$, so that $M\otimes N\in \Img(U) = \Img(F^*)$. Thus $\Img(F^*)$ is a tensor ideal in $\Rep \cat C$.
\end{proof}

\begin{Rem}
A slightly weaker version of \Cref{Thm:main2-intro} is proved in \cite[\S1.4]{Huglo19pp} by way of more explicit calculations.
We do not know if the ridigity hypothesis is necessary for $E_A$ to be an equivalence, nor if the fullness hypothesis is necessary for $E_A$ to be a \emph{strong} tensor functor (it is always a lax one), as we have not looked for explicit counterexamples. 
On the other hand, we don't see any reason for the conclusions to hold otherwise, because the projection formula of \Cref{Prop:proj-formula}, in particular, may fail.
\end{Rem}

\section{Application: cohomological vs ordinary Mackey functors}
\label{sec:cohomack}%

In this section we derive \Cref{Cor:cohomack} from the above abstract results. 
Although this corollary is essentially well-known (see especially~\cite[\S10]{PanchadcharamStreet07}), the present proof offers some insight because it immediately clarifies the relation between ordinary and cohomological Mackey functors as tensor categories. Besides, it provides an easier and probably more familiar analogue of our main application, \Cref{Cor:bisetfun}.  

Throughout this section, we fix a finite group~$G$ and a commutative ring~$\kk$.

\begin{Rec}[The span category of $G$-sets]
\label{Rec:span}
Let $G\sset$ denote the category of finite left $G$-sets. Then there is a category $\spancat_\kk (G)$ whose objects are the same as those of $G\sset$, and where a morphism $X\to Y$ is by definition an element of the Grothendieck group $\kk \otimes_\mathbb Z K_0(G\sset / X\times Y)$  (with addition induced by coproducts) of the slice category $G\sset / X\times Y$. 
In particular, every morphism can be written as a finite $\kk$-linear combination of isomorphism classes $[\alpha,\beta]$ of spans $X \overset{\;\alpha}{\gets} S\overset{\beta}{\to} Y$ in $G\sset$, where two spans $X \overset{\;\alpha}{\gets} S\overset{\beta}{\to} Y$ and $X \overset{\;\alpha'}{\gets} S' \overset{\beta'}{\to} Y$ are \emph{isomorphic} if there exists an isomorphism $\varphi\colon  S\overset{\sim}{\to} S'$ making the following diagram of $G$-sets commute:
\[
\xymatrix@R=8pt@C=10pt{
&& S \ar[dd]_\varphi^\simeq \ar[dll]_\alpha \ar[drr]^\beta && \\
X &&  && Y \\
&& S' \ar[ull]^{\alpha'} \ar[urr]_{\beta'} &&
}
\]
A $\kk$-bilinear composition in $\spancat_\kk (G)$ is induced by taking pull-backs in~$G\sset$:
\begin{equation}
\label{eq:Sp(G)-comp}
\vcenter{
\xymatrix@R=11pt@C=12pt{
&& S\times_Y T  \ar[dl]^-{\tilde \gamma} \ar[dr]_-{\tilde \beta} \ar@/_3ex/[ddll]_-{\alpha \tilde \gamma} \ar@/^3ex/[ddrr]^-{\delta \tilde \beta}   &&  \\
& S \ar[dl]^-\alpha \ar[dr]_-\beta && T \ar[dl]^-\gamma \ar[dr]_-\delta & \\
X 
\ar@{..>}[rr] && Y \ar@{..>}[rr] && Z
}}
\end{equation}
It follows that $\spancat_\kk(G)$ is an essentially small $\kk$-linear category, where the sum of two spans is induced by taking coproducts at the middle object~$S$.
Moreover,  $\spancat_\kk(G)$ is a $\kk$-linear rigid tensor category, with tensor product $\otimes$ induced by the categorical product $X\times Y$ of $G\sset$ and with tensor unit $\unit = G/G$ the one-point $G$-set. 
Every $G$-set is actually its own tensor dual, by virtue of the natural equivalence
$G\sset / X\times Y \simeq G \sset / Y \times X$
of slice categories. See details \eg in \cite{Bouc97}.
\end{Rec}

\begin{Rec}[The category of permutation modules]
\label{Rec:perm}
We denote by $\Perm_\kk (G)$ the category of finitely generated \emph{permutation $\kk G$-modules}, that is, the full subcategory of those (left) $\kk G$-modules which admit a finite $G$-invariant $\kk$-basis. 
Clearly this is an essentially small $\kk$-linear category. It is moreover a $\kk$-linear tensor category, because it inherits the usual tensor product of $\kk G$-modules (\ie the tensor product $\otimes = \otimes_\kk$ over~$\kk$ endowed with diagonal $G$-action). Indeed, the trivial module $\unit = \kk$ is a permutation $\kk G$-module, and if the $\kk G$-modules $M$ and $N$ admit $G$-invariant bases $X\subset M $ and $ Y\subset N$, respectively, then $\{x\otimes y\mid (x,y)\in X\times Y\}$ is a $G$-invariant basis of $M\otimes N$. 
As tensor category, $\Perm_\kk (G)$ is rigid: if $M$ has invariant basis~$X$, its tensor-dual module $M^\vee = \Hom_\kk (M,\kk)$ has an invariant basis given by the usual $\kk$-linear dual basis $X^\vee := \{ x^\vee\colon y\mapsto \delta_{x,y} \mid x\in X\}$.
\end{Rec}

\begin{Def}[Mackey functors for~$G$]
\label{Def:MackcoMack}
The representation category
\[ \Mackey_\kk(G) := \Rep \spank (G) \]
 (\cf \Cref{Not:Funk}) is by definition the category of \emph{\textup($\kk$-linear\textup) Mackey functors for~$G$}. 
 Similarly, the category of \emph{cohomological} Mackey functors for~$G$ is
\[ \coMackey_\kk(G) := \Rep \Perm_\kk (G) \,. \]
Both are complete and co-complete abelian $\kk$-linear tensor categories, with tensor structure provided by Day convolution (\Cref{Cons:Day-convo}) extended from the rigid tensor structures on their respective source categories $\spank (G)$ and $\Perm_\kk(G)$, as described in \Cref{Rec:span} and \Cref{Rec:perm}.
\end{Def}

\begin{Rem} \label{Rem:box-prods}
Usually, the tensor product in $\Mackey_\kk(G)$ is denoted by $M \boxprod N$ and called \emph{box product};
and  $\coMackey_\kk(G)$ is defined as a full subcategory of $\Mackey_\kk(G)$ and only later identified (by Yoshida's theorem) with the above functor category.
\end{Rem}

\begin{Lem}[Yoshida's functor]
\label{Lem:Yo}
There is a well-defined $\kk$-linear tensor functor 
\[
\Yo \colon \spancat_\kk(G) \longrightarrow \Perm_\kk(G)
\]
sending a left $G$-set $X$ to the permutation $\kk G$-module~$\kk [X]$ and a span $X \overset{\;\alpha}{\gets} S\overset{\beta}{\to} Y$ of $G$-maps to the $\kk G$-linear map $\kk [X]\to \kk [Y]$, $x \mapsto \sum_{s \in \alpha^{-1}(x)} \beta(s)$.
Moreover, $\Yo$ is essentially surjective and full.
\end{Lem}

\begin{proof}
One can verify that $\Yo$ is well-defined by straightforward computations, but here is a more conceptual way to see it.
First note that there is a functor $\Yo_\star\colon G\sset\to \kk\MMod$ which sends $X$ to $\kk[X]$ and simply extends a $G$-map $\kk$-linearly.
There is also a functor $\Yo^\star\colon (G\sset)^\op \to \kk\MMod$ with the same object-map but sending a $G$-map $\alpha\colon Y\to X$ to the $\kk$-linear map $\kk[X]\to \kk[Y]$ such that $x \mapsto \sum_{y \in \alpha^{-1} (x)} y$ for $x\in X$. 
Since $\kk[ X\sqcup Y] \simeq \kk [X] \oplus \kk[Y]$, and since the identity 
\[ \Yo^\star(\gamma) \Yo_\star(\beta)= \Yo_\star(\tilde \beta)\Yo^\star(\tilde \gamma)\]
can be readily verified for an arbitrary pull-back square of $G$-sets (with notations as in~\eqref{eq:Sp(G)-comp}),
it follows by the universal property of the span category (\cite{Lindner76}; see also \cite[App.\,A.5]{BalmerDellAmbrogio20}) that there is a functor $\tilde{\Yo}\colon \spancat (G)\to \Perm_\kk(G)$ defined as $\tilde{ \Yo}(X)= \kk[X]$ on $G$-sets and  $\tilde{ \Yo} ([\alpha, \beta]) = \Yo_\star(\beta)\circ \Yo^\star(\alpha)$ on spans of $G$-maps. 
(Here $\spancat(G)$ is the `plain' span category, constructed as $\spancat_\kk(G)$ but with Hom sets simply given by the sets of isomorphism classes $\spancat(G)(X,Y)=(G\sset / X\times Y)/_\simeq$ of objects.) Our functor $\Yo$ is then the evident $\kk$-linear extension of~$\tilde{ \Yo}$. 

A permutation $\kk G$-module $M$ is precisely one for which there exists an isomorphism $M \simeq \kk[X]$ for some $G$-set~$X$, hence $\Yo$ is essentially surjective. 
It is a little harder to see that $\Yo$ is full, but it suffices to verify it for two standard orbits $X=G/H$ and $Y=G/K$, in which case it boils down to the $\kk$-linear isomorphism
\[
\kk [H \!\setminus\! G / K] \overset{\sim}{\longrightarrow} \Hom_{\kk G} ( \kk [G/H], \kk [G/K]) , \;\; HxK \mapsto \Big(gH \mapsto \sum_{[u] \in H/(H\cap {}^xK)} guxK\Big)
\]
as in \cite[Lemma~3.1]{Yoshida83b}; see \cite[Prop.\,2.1.10]{Huglo19pp} for the remaining details.
\end{proof}

\begin{proof}[Proof of \Cref{Cor:cohomack}]
As we have seen in Recollections~\ref{Rec:span} and~\ref{Rec:perm}, both $\cat C:= \spancat_\kk(G)$ and $\cat D:= \Perm_\kk (G)$ are essentially small rigid $\kk$-linear tensor categories.
Moreover, Yoshida's functor $F:=\Yo$ of \Cref{Lem:Yo} is a full and essentially surjective $\kk$-linear tensor functor between them.
Hence $F\colon \cat C\to \cat D$ satisfies all the hypotheses of \Cref{Thm:main2-intro}, from which we obtain most of the claims of the corollary.

It remains only to clarify two points:
\begin{enumerate}[(1)]
\item The monoid $A=F^*(\unit)$ of $\Rep \cat C=\Mackey_\kk(G)$ from the theorem is the \emph{fixed-point Mackey functor}~$\FPk$. 
\item A Mackey functor $M$ is cohomological if and only if it satisfies the cohomological relations $\mathrm{ind}_L^H \circ \mathrm{res}^H_L=[H:L] \, \id_{M(H)} $ for all subgroups $L\leq H\leq G$.
\end{enumerate}
Here for familiarity we have switched to the classical notations $M(H):= M(G/H)$, $\mathrm{ind}_L^H:= M([G/L = G/L \to G/H])$ and $\mathrm{res}^H_L:=M([G/H \gets G/L = G/L])$, where $G/L \to G/H$ is the quotient $G$-map for two nested subgroups $L\leq H\leq G$.

For~(1) recall that, classically, $\FPk$ is the Mackey functor which assigns to every orbit $G/H$ the trivial $\kk G$-module~$\kk$, whose restriction and conjugation maps are all identities, and whose induction maps $\mathrm{ind}_L^H\colon \kk \to \kk$ are given by multiplication by the index~$[H:L]$.
It is a matter of straightforward comparison to identify it with $A = \Perm_\kk(G)(\kk, \Yo(-))$ (if necessary, see details in \cite[Lemma 2.2.15]{Huglo19pp}).

For~(2), the easiest way to see this is as in \cite[Prop.\,16.3]{ThevenazWebb95} where, by very easy explicit calculations, it is shown that a Mackey module over the Green functor $\FPk$ is the same thing as a Mackey functor satisfying also the cohomological relations.
\end{proof}

\begin{Rem}
Note that (2) amounts to saying that the kernel (on maps) of Yoshida's functor $\Yo$ is generated as a $\kk$-linear categorical ideal of $\spank (G)$ by the span-versions of the cohomological relations, namely (after computing the trivial pull-back) by 
\[
[G/H \gets G/L \to G/H] - [H:L] \id_{G/H} \quad \textrm{ for all } L\leq H \leq G. 
\]
It is immediate to see that $\Yo$ kills these relations. To see that they actually generate the kernel of~$\Yo$, it suffices to inspect the standard presentation of $\spank (G)$ in terms of restriction, conjugation and induction maps; the necessary calculations are essentially a re-writing of the ones in \cite[Prop.\,16.3]{ThevenazWebb95} we mentioned above.
\end{Rem}

\section{The realization pseudo-functor}
\label{sec:R}%

This section is dedicated to proving \Cref{Thm:main1-intro}.

We retain the same standard notations and conventions for bicategories, 2-categories, pseudo-functors and allied notions as in \cite[App.\,A]{BalmerDellAmbrogio20} or \cite[\S2]{DellAmbrogio19pp}. 
We nonetheless provide here a few recollections for the reader's convenience.
We denote by $\gpd$ the 2-category ($=$ strict bicategory) of finite groupoids, functors between them and (necessarily invertible) natural transformations.

\begin{Ter} \label{Ter:squares}
Given two functors $S\overset{a}{\to} G \overset{b}{\gets} T$  between finite groupoids and with common target (a `cospan'), we can build its \emph{iso-comma} groupoid $a/b$, whose objects are triples $(s,t,\gamma)$ with $s\in \Obj S$, $t\in \Obj T$ and $\gamma \in G(a(s),b(t))$. 
A morphism $(s,t,\gamma)\to (s',t',\gamma')$ is a pair $(\varphi,\psi)$ with $\varphi \in S(s,s')$ and $\psi \in T(t,t')$ and such that $\gamma'a(\varphi)=b(\psi)\gamma$.
It is part of the \emph{iso-comma square}
\[
\vcenter{
\xymatrix@C=12pt@R=12pt{
& (a/b) \ar[dl]_p \ar[dr]^q \ar@{}[dd]|{\underset{\sim}{\Ivocell{\gamma}}} & \\
S \ar[dr]_a &  & T \ar[dl]^b \\
& G &
}}
\]
which also comprises two evident forgetful functors $p,q$ and a tautological natural isomorphism $\gamma \colon ap\Rightarrow bq$ whose component at the object $(s,t,\gamma)$ is the map~$\gamma$. The iso-comma square is the universal (in a strict 2-categorical sense) such invertible 2-cell sitting over the given cospan.
\end{Ter}

\begin{Cons}[{The bicategory $\Span$}]
\label{Cons:Span}
There exists a bicategory $\Span$ consisting of the following data. 
Its objects are all the finite groupoids. A 1-cell $H\to G$ is a `span' in~$\gpd$, that is a diagram
\[
\xymatrix{ H & S  \ar[l]_-{b} \ar[r]^-a & G}
\]
of functors between finite groupoids. 
A 2-cell from $H \overset{\;b}{\gets} S \overset{a}{\to} G$ to $H \overset{\;\;b'}{\gets} S' \overset{a'}{\to} G$ is the isomorphism class of a diagram in~$\gpd$ of the following form:
\[
\vcenter{
\xymatrix@R=10pt@C=12pt{
&& S \ar[dd]_(.3)f \ar[dll]_b \ar@{}[ddl]|{\beta \SEcell \;\;} \ar@{}[ddr]|{\;\;\NEcell \alpha} \ar[drr]^a && \\
H &&  && G \\
&& S' \ar[ull]^{b'} \ar[urr]_{a'} &&
}}
\]
(The orientation of the two 2-cells is merely a matter of convention.) Here, two such diagrams are \emph{isomorphic} if there exists a natural isomorphism between their 1-cell components $f$ which identifies their 2-cells components $\alpha$ and~$\beta$. The horizontal composition of spans is defined by constructing an iso-comma square in the middle: 
\[
\xymatrix@R=11pt@C=12pt{
&& (c/b)  \ar[dl]^-p \ar[dr]_-q \ar@/_3ex/[ddll]_-{dp} \ar@/^3ex/[ddrr]^-{aq}  \ar@{}[dd]|{} &&  \\
& T \ar@{}[rr]|{\Ivocell{\gamma}} \ar[dl]^-d \ar[dr]_-c && S \ar[dl]^-b \ar[dr]_-a & \\
K 
\ar@{..>}[rr] && H \ar@{..>}[rr] && G
}
\]
The horizontal composition of 2-cells, as well as the coherent associativity and unitality isomorphisms, are all induced by the universal property of iso-comma squares in a straightforward way. The identity span of $G$ is $\Id_G = (G = G = G )$. See \cite[\S\,5.1]{BalmerDellAmbrogio20} for more details.
\end{Cons}

In the following,  $(-)^\co$ and $(-)^\op$ denote, respectively, the operation of formally reversing the direction of the 2-cells or of the 1-cells in a bicategory.

\begin{Cons}[Canonical embeddings]
\label{Cons:can-emb}
There are two canonical pseudo-functors $(-)_!\colon \gpd^\co\hookrightarrow \Span$ and $(-)^*\colon \gpd^\op\hookrightarrow \Span$, embedding $\gpd$ inside of~$\Span$ in a way which is contravariant on 2-cells and on 1-cells, respectively: 
\[
\vcenter{
\xymatrix{
& \\
S \ar@{}[r]|{\Ncell\,\alpha} \ar@/^3ex/[r]^-a \ar@/_3ex/[r]_-{a'} &  G \\
&
}}
\quad\longmapsto\quad 
\alpha_! =
\left[
\vcenter{
\xymatrix@R=10pt@C=12pt{
&& S \ar@{=}[dd] \ar@{=}[dll] \ar@{}[ddl]|{\id \SEcell \;\;} \ar@{}[ddr]|{\;\;\NEcell \alpha} \ar[drr]^a && \\
S &&  && G \\
&& S \ar@{=}[ull] \ar[urr]_{a'} &&
}}
\right]
\]
\[
\vcenter{
\xymatrix{
& \\
H \ar@{}[r]|{\Scell\,\beta} \ar@/^3ex/@{<-}[r]^-b \ar@/_3ex/@{<-}[r]_-{b'} &  S \\
&
}}
\quad\longmapsto\quad 
\beta^* =
\left[
\vcenter{
\xymatrix@R=10pt@C=12pt{
&& S \ar@{=}[dd] \ar[dll]_-b \ar@{}[ddl]|{\beta \SEcell \;\;} \ar@{}[ddr]|{\;\;\NEcell \id} \ar@{=}[drr] && \\
H &&  && G \\
&& S \ar[ull]^{b'} \ar@{=}[urr] &&
}}
\right]
\]
(the above diagrams to be understood in~$\gpd$). 
Thus the embeddings map functors $a\colon S\to G$ and $b\colon S\to H$ to spans $a_!=(S = S \overset{a}{\to} G)$ and $b^*=(H \overset{\;b}{\gets} S = S)$, respectively,
and natural isomorphisms $\alpha\colon a' \Rightarrow a$ and $\beta\colon b \Rightarrow b'$ to the depicted 2-cells.
Note that these pseudo-functors are not strict. Every 1-cell of $\Span$ is (isomorphic to) a composite $a_! \circ b^*$, and similarly, every 2-cell $[f,\beta, \alpha]$ is a combination of  $\alpha_!$ and~$\beta^*$.
See \cite[Cons.\,5.1.18, Rem.\,5.1.19, Prop.\,5.1.32]{BalmerDellAmbrogio20} for details.
\end{Cons}

The key tool for defining the realization pseudo-functor $\cat R$ is the universal property of its source, the span bicategory: 

\begin{Thm} [Universal property of $\Span$] 
\label{Thm:UP-Span}
Suppose we are given a bicategory~$\cat B$, two pseudo-functors
\begin{equation*}
\cat F_! \colon \gpd^\co \to \cat B 		\quad\quad \textrm{ and } \quad\quad 		\cat F^*  \colon \gpd^\op\to \cat B
\end{equation*}
and, for every functor $u\colon H\to G$ between finite groupoids, an (internal) adjunction
$\cat F_!(u) \dashv \cat F^*(u)$
in~$\cat B$, with specified unit and counit.
Assume the following holds:
\begin{enumerate}[\rm (a)]
\item On objects, $\cat F_!$ and $\cat F^*$ coincide: $\cat F_!(G)= \cat F^*(G)$ for all~$G$.
\item The adjunctions satisfy base-change, a.k.a.\ the Beck-Chevalley condition, for all iso-comma squares (\Cref{Ter:squares}). In other words, for every iso-comma square in $\gpd$ as on the left
\begin{equation} \label{eq:BCmate}
\vcenter{
\xymatrix@C=12pt@R=12pt{
& (a/b) \ar[dl]_p \ar[dr]^q \ar@{}[dd]|{\underset{\sim}{\Ivocell{\gamma}}} & \\
S \ar[dr]_a &  & T \ar[dl]^b \\
& G &
}}
\quad\quad \rightsquigarrow \quad\quad
\vcenter{
\xymatrix@C=10pt@R=12pt{
&&& \cat FT \ar@{=}@/^4ex/[dd] \ar@{}[dd]|{\overset{\scriptstyle \varepsilon}{\Ecell}} \\
&& \cat F(a/b) \ar[ur]^{\cat F_!q}  & \\
& \cat FS \ar[ur]^-{\cat F^*p} \ar@{}[dd]|{\Ivocell{\eta}} \ar@{}[rr]|{\overset{\scriptstyle \cat F^*\gamma}{\Ecell}\;\;} & &\cat FT \ar[ul]^(.45){\cat F^*q\!\!} \\
&& \cat FG \ar[ul]^(.45){\cat F^*a\!\!\!} \ar[ur]_-{\cat F^*b} & \\
& \cat FS \ar[ur]_-{\cat F_!a} \ar@{=}@/^4ex/[uu] &&
}}
\end{equation}
the mate constructed on the right is an isomorphism $\cat F_!(q)\cat F^*(p)  \overset{\sim}{\Rightarrow}  \cat F^*(b) \cat F_!(a)$. 
\item For each 2-cell $\alpha \colon u\Rightarrow v$ in~$\gpd$, the 2-cells $\cat F_!(\alpha)$ and $\cat F^*(\alpha)$ of~$\cat B$ are each other's mate under the adjunctions $\cat F_!(u)\dashv \cat F^*(u)$ and $\cat F_!(v)\dashv \cat F^*(v)$ (after a necessary inversion). 
Similarly, the coherent structure isomorphisms of $\cat F_!$ and $\cat F^*$ are each other's mates  (in the only way which makes sense).
\end{enumerate}
Then the above data defines a pseudo-functor $\cat F\colon \Span \to \cat B$ by the composite
\[
\cat F \big( H \overset{b}{\gets} S \overset{a}{\to} G \big) :=  \cat F_!(a) \circ \cat F^*(b)
\]
for 1-cells and by the pasting 
\[
\cat F
\left( \left[
\vcenter{
\xymatrix@R=12pt{
& S \ar[dl]_b \ar[dr]^a \ar[dd]^(.3)f & \\
H & \ar@{}[l]|(.4){\beta\;\SEcell} \ar@{}[r]|(.4){\NEcell\;\alpha} & G \\
& S' \ar[ul]^{b'} \ar[ur]_{a'} &
}}
\right] \right)
\;\; := \;\;
\vcenter{
\xymatrix@R=15pt@C=15pt{
\cat F H \ar@{}[rr]|{\Scell\; \cat F^*\beta} \ar[dr]_{\cat F^*b'} \ar@/^4ex/[rr]^-{\cat F^*b} &&
 \cat F S \ar@{}[rr]|{\Scell\; \cat F_!\alpha} \ar[dr]| (.4) {\cat F_!f} \ar@/^4ex/[rr]^-{\cat F_!a } && \cat F G \\
& \cat F S' \ar@{}[rr]|{\Scell\; \varepsilon} \ar[ur]|{\cat F^*f} \ar@/_4ex/@{=}[rr] &&
 \cat FS' \ar[ur]_{\cat F_!a'} & \\
&&&&
}}
\]
for 2-cells. 
This $\cat F$ is, up to isomorphism, the unique pseudo-functor $\Span \to \cat B$ such that $\cat F_! \simeq \cat F\circ (-)_!$ and $\cat F^* \simeq \cat F\circ (-)^*$.
\[
\xymatrix{ 
\gpd^\co \ar@{}[dr]|{\simeq} \ar[d]_{(-)_!} \ar@/^2ex/[rrd]^{\cat F_!} && \\
\Span \ar@{-->}[rr]^-{\cat F} && \cat B \\
\gpd^\op \ar@{}[ur]|{\simeq} \ar[u]^{(-)^*} \ar@/_2ex/[urr]_{\cat F^*} && 
}
\]
\end{Thm}

\begin{proof}
This is a special case, and a slight rephrasing which emphasizes the symmetry of the present situation, of the more general \cite[Theorem~5.2.1]{BalmerDellAmbrogio20}. 
Indeed, by hypotheses (a) and (b) we can apply \emph{loc.\,cit.} with $\GG = \JJ = \gpd$ to the pseudo-functor $\cat F^*$. (To be precise, \emph{loc.\,cit.} assumes that the target bicategory is strict, so we should first replace $\cat B$ with a biequivalent 2-category~$\cat C$; this has the effect that we can only obtain an isomorphism $\cat F^* \simeq \cat F\circ (-)^*$ rather than an equality; \cf \cite[Theorem~5.3.7]{BalmerDellAmbrogio20}). We thus obtain an extension~$\cat F\colon \Span\to \cat B$, constructed as in the theorem with the only (possible) difference that, in the pasting defining the image of the 2-cell~$[f,\beta,\alpha]$, the 2-cell $\cat F_!(\alpha)$ must be replaced by the mate 
\[
\xymatrix@C=18pt@L=4pt{
 \cat F_!(a) \ar@{=>}[r]^-{\eta} &
  \cat F_!(a) \cat F^*(a'f) \cat F_!(a'f) \ar@{=>}[r]^-{\cat F^*\alpha} &
   \cat F_!(a) \cat F^*(a) \cat F_!(a'f) \ar@{=>}[r]^-\varepsilon &
    \cat F_!(a'f) \simeq \cat F_!(a') \cat F_!(f)
 }
\]
of~$\cat F^*(\alpha)$. This $\cat F$ is such that $\cat F^* \simeq \cat F \circ (-)^*$ and is unique up to an isomorphism of pseudo-functors for this property. By its construction, it is uniquely determined (on the nose) by the pseudo-functor $\cat F^*$, the isomorphism $\cat F^* \simeq \cat F \circ (-)^*$, and by taking mates with respect to the given adjunctions $\cat F_!(u)\dashv \cat F^*(u)$ for all~$u$.

It remains to see that we also have $\cat F_!\simeq \cat F \circ (-)_!$, and that the difference in the definition of 2-cells is only apparent.
These however are straightforward consequences of the construction of $\cat F$ together with hypothesis~$(c)$.
\end{proof}

We next recall our bicategory $\cat B$ of interest and proceed to introduce the structure needed to apply the universal property of~$\Span$.

\begin{Cons}[The bicategory $\Biset$]
\label{Cons:Biset}
There exists a bicategory $\Biset$ consisting of the following data. 
Its objects are all finite groupoids. 
A 1-cell $U\colon H\to G$ (a~`finite left-$G$ and right-$H$ biset', or `$G,H$-biset' for short) is a functor 
\[
U\colon H^\op\times G\longrightarrow \set
\]
taking values in the category of finite sets.  
A 2-cell $\varphi\colon U\Rightarrow V$ is a natural transformation $U\Rightarrow V$. 
The horizontal composition of two composable bisets $V\colon K\to H$ and $U\colon H\to G$ is given by the set-theoretical coend (see \Cref{sec:preliminaries})
\[
U \circ V = U \otimes_H V : = \int^{h\in H} U(h, - ) \times V(- , h) \colon \quad K^\op \times G \longrightarrow G \,.
\]
The identity 1-cell of a groupoid $G$ is its Hom functor $\Id_G=G(-,-)\colon G^\op\times G\to \set$. 
The horizontal composition of 2-cells is induced on the quotient sets in the evident way.
The coherent associativity and unitality constraints are the standard (evident) identifications of coends.
See \eg \cite[\S 7.8]{Borceux94a} for details.
\end{Cons}

\begin{Not} \label{Not:Rs}
Let $u\colon H\to G$ be any functor of finite groupoids. 
We will write 
\begin{align*}
\cat R_!(u):= G(u-,-) \colon H^\op\times G\to \set \\ 
\cat R^*(u):= G(- , u-)\colon G^\op\times H\to \set
\end{align*}
for the bisets $H\to G$ and $G\to H$ obtained by composing $u$ with the Hom functor of~$G$ in the two possible ways.
\end{Not}

\begin{Lem} \label{Lem:R-pseudo-funs}
The assignments $u\mapsto \cat R_!(u)$ and $u\mapsto \cat R^*(u)$ of \Cref{Not:Rs} extend canonically to two pseudo-functors 
\[ \cat R_!\colon \gpd^\co \to \Biset 
\quad \textrm{ and } \quad
\cat R^*\colon \gpd^\op \to \Biset 
\]
both of which act as the identity on objects (finite groupoids).
\end{Lem}

\begin{proof}
Let us specify this structure for~$\cat R^*$. By definition, $\cat R^*$ sends a groupoid $G$ to itself and (contravariantly) a functor $u\colon H\to G$ to the biset $\cat R^*(u)=G(-,u-)\colon G\to H$.
For a natural transformation $\alpha \colon u\Rightarrow v$, we naturally define the (covariant!) image $\cat R^*(\alpha)\colon \cat R^*(u)\Rightarrow \cat R^*(v)$ to be the natural isomorphism $G(-,u-) \Rightarrow G(-,v-)$ induced by~$\alpha$, by sending an element $\xi$ to $\alpha \circ \xi$. 
The assignment $\alpha \mapsto \cat R^*(\alpha)$ defines a functor for each pair $(H,G)$, as required.
The structure isomorphisms of the pseudo-functor are given by the identity map $\un_{\cat R^*}\colon \Id_{\cat R^*(G)} = G(-,-) = \cat R^*(\Id_G)$ for each~$G$ as unitor, and by the isomorphism
\begin{align*}
\fun_{\cat R^*}\colon \cat R^*(v) \circ \cat R^*(u)  = \int^{h\in H}  H(h, v-) \times G(-,uh)  \overset{\sim}{\longrightarrow} G(- , uv -) = \cat R^*(u\circ v)
\end{align*}
induced by composition, $[ \zeta , \xi ]_h \mapsto u(\zeta) \circ \xi$, for composable functors $K \overset{v}{\to} H \overset{u}{\to} G$; this map is clearly well-defined with inverse given by $\xi \mapsto [\id,\xi]$. The verification of the  coherence axioms is straightforward and left to the reader.

Similarly for~$\cat R_!$, a 2-cell $\alpha\colon u\Rightarrow v$ is sent (contravariantly) to the natural map $G(v-,-) \Rightarrow H(u-,-)$ given by \emph{pre}composition with $\alpha$, that is $\xi \mapsto \xi\circ \alpha$, and the structural isomorphism 
\begin{align*}
\fun_{\cat R_!}\colon \cat R_!(u) \circ \cat R_!(v)   = \int^{h\in H}   G(uh,-) \times H(v-, h)  \overset{\sim}{\longrightarrow} G(uv- , -) = \cat R_! (u\circ v)
\end{align*}
is again simply given by composition: $[\xi, \zeta]_h\mapsto \xi \circ u(\zeta)$.
\end{proof}

\begin{Lem}  \label{Lem:R-adj}
For every $u\colon H\to G$, there is an adjunction $\cat R_! (u)\dashv \cat R^*(u)$ in the bicategory $\Biset$, with the natural transformations 
\begin{align*}
\eta_u \colon \Id_H &\Longrightarrow \cat R^*(u)\circ \cat R_!(u)   &    \varepsilon_u \colon \cat R_!(u) \circ \cat R^*(u) &\Longrightarrow \Id_G \\
\zeta &\longmapsto [\id, u(\zeta)]   	&     [\xi', \xi ] &\longmapsto \xi'\xi
\end{align*}
providing the unit and counit, respectively.
\end{Lem}

\begin{proof}
It is straightforward to verify that these are well-defined maps satisfying the zig-zag equations of an adjunction. 
For the latter,  at each object $(g,h)\in G^\op\times H$ we may follow an element $\xi \in \cat R^*(u)(g,h)=G(g,uh)$ through the composite
\[
\xymatrix@C=7pt@R=16pt@L=6pt{
\cat R^*(u) \ar@{=>}[d]^{\simeq} &&
 G(g,uh) \ar[d]  && 
   \xi \ar@{|->}[d] \ar@{}[ll]|-{\ni} \\
\Id_H \; \cat R^*(u) \ar@{=>}[d]_{\eta \, \circ \, \id} &&
  H(h, h ) \times G(g,uh) \ar[d] &&
   [\id, \xi ] \ar@{|->}[d] \\
\cat R^*(u) \; \cat R_!(u) \; \cat R^*(u) \ar@{=>}[d]_{\id \, \circ \, \varepsilon } &&
 G(uh, uh ) \times G(uh, uh ) \times G(g,uh)  \ar[d] &&
  [\id, u (\id), \xi] \ar@{|->}[d] \\
\cat R^*(u) \; \Id_G \ar@{=>}[d]^{\simeq} &&
 G(uh, uh) \times G(g, uh) \ar[d] &&
  [\id, u(\id) \circ \xi] \ar@{|->}[d] \\
\cat R^*(u) &&
 G(g, uh) &&
\id\circ  u(\id) \circ \xi = \xi
}
\]
which is thus shown to be the identity map, as required. 
In the above display, the middle column shows to which sets belong the representatives of the coend elements displayed on the right-hand column, before quotienting. The left and right unitors in $\Biset$ are induced by composition of maps, like~$\varepsilon$, with inverse given by insertion of an identity map. 

The verification of the other zig-zag equation is similar. 
\end{proof}

\begin{Lem} \label{Lem:R-BC}
For the adjunctions $\cat R_!(u)\dashv \cat R^*(u)$ of \Cref{Lem:R-adj}, the mate 
of  every iso-comma square $\gamma$ as in \eqref{eq:BCmate} is an invertible 2-cell in $\Biset$.
\end{Lem}

\begin{proof}
This is another direct computation, although rather more involved. 
Unfolding the construction of the mate~$\cat R^*(\gamma)_!\colon \cat R_!(q)\circ  \cat R^*(p) \Rightarrow \cat R^*(b) \circ \cat R_!(a)$, we obtain the following composite natural transformation on the left-hand side (where, as before, we omit the associativity constraints of~$\Biset$):
\[
\xymatrix@C=-3pt@R=16pt@L=6pt{
\cat R_!(q) \; \cat R^*(p) \ar@{=>}[d]^{\simeq} &&
 T(y,t) \!\times\! S(s,x) \ar[d] && 
   [\tau, \sigma] \ar@{|->}[d] \\
\cat R_!(q) \; \cat R^*(p) \; \Id_S \ar@{=>}[d]_{\id \,\circ\, \id \,\circ\, \eta} &&
 T(y,t) \!\times\! S(s,x) \!\times\! S(s,s)\ar[d]  &&
   [\tau, \sigma,\id] \ar@{|->}[d] \\
\cat R_!(q) \; \cat R^*(p) \; \cat R^*(a) \; \cat R_!(a) \ar@{=>}[d]_{\id \,\circ\, \fun_{\cat R^*} \,\circ \, \id}^\simeq &&
   T(y,t) \!\times\! S(s,x) \!\times\! G(as,as) \!\times\! G(as,as) \ar[d] &&
    [\tau, \sigma,\id, \id] \ar@{|->}[d]  \\
\cat R_!(q) \; \cat R^*(ap) \; \cat R_!(a) \ar@{=>}[d]_{\id \,\circ\, \gamma \,\circ \, \id} &&
    T(y,t) \!\times\! G(as,ax) \!\times\! G(as,as) \ar[d] &&
     [\tau,a(\sigma), \id] \ar@{|->}[d]  \\
\cat R_!(q) \; \cat R^*(bq) \; \cat R_!(a) \ar@{=>}[d]_{\id \,\circ\, \fun^{-1}_{\cat R^*} \,\circ \, \id}^\simeq &&
    T(y,t) \!\times\! G(as,bx) \!\times\! G(as,as) \ar[d] &&
      [\tau, \gamma a(\sigma), \id]  \ar@{|->}[d]  \\
\cat R_!(q) \; \cat R^*(q) \; \cat R^*(b) \; \cat R_!(a) \ar@{=>}[d]_{\varepsilon \,\circ\, \id \,\circ\, \id} &&
    T(y,t) \!\times\! T(y,y) \!\times\! G(as,by) \!\times\! G(as,as) \ar[d] &&
      [\tau, \id, \gamma a(\sigma), \id] \ar@{|->}[d] \\
\Id_T \; \cat R^*(b) \; \cat R_!(a) \ar@{=>}[d]^{\simeq} &&
    T(y,t) \!\times\! G(as,by) \!\times\! G(as,as) \ar[d] &&
     [\tau, \gamma a(\sigma), \id] \ar@{|->}[d] \\
\cat R^*(b) \; \cat R_!(a) &&
    G(as,bt) \!\times\! G(as,as) && 
     [b(\tau) \gamma a(\sigma), \id]
}
\]  
At each object $(s,t)\in S^\op\times T$, we can follow the trajectory of an arbitrary element $[\tau, \sigma]_i \in (\cat R_!(q) \circ \cat R^*(p)) (s,t)$, as indicated on the right-hand side. 
Here $(\tau, \sigma)\in T(qi , t) \times S(s, pi)$ for some object $i = (x,y,\gamma\colon a(x)\to b(y))\in (a/b)$, so that $p(i)=x$ and $q(i)=y$.
The structural isomorphism $\fun_{\cat R^*}$ and its inverse are as in the proof of \Cref{Lem:R-pseudo-funs} (again, given by composition and insertion of an identity).

It remains to see that the resulting map above
\begin{align*}
\int^{i \in (a/b)} T(qi , t) \times S(s,pi ) &\longrightarrow \int^{g\in G} G(g, bt) \times G(as ,g)  \\
 [\tau,\sigma]_i   & \longmapsto [\,b(\tau) \gamma a(\sigma)\,,\, \id\,]_{as}
\end{align*}
is a bijection. 

It is injective, because if $[\tau',\sigma']_{i '}$ (for some $i'=(x',y',\gamma')$) is such that we have $[b(\tau') \gamma' a (\sigma'), \id]=[b(\tau) \gamma a(\sigma), \id]$  in the target coend, then (using that $G$ is a groupoid) there exists a $\varphi\colon as \to as'$ in $G$ such that
$
b(\tau') \gamma' a (\sigma') \circ \varphi =  b(\tau) \gamma a (\sigma)
$ and $
\varphi \circ \id_{as} = \id_{as}
$,
and therefore
\[ 
b(\tau') \gamma' a (\sigma')  =  b(\tau) \gamma a (\sigma) \,.
\]
The latter condition states precisely that the pair 
$(\tau'^{-1}\tau , \sigma'\sigma^{-1})\in T(y,y') \times S(x,x')$ 
defines a map $i \to i'$ in $a/b$, showing that $[\tau, \sigma]_i=[\tau',\sigma']_{i'}$ in the source coend. 

To see the map is surjective, let $(\zeta, \xi)\in G(g,bt)\times G(as, g)$ represent an arbitrary element of the target coend.
Then $i:= (s,t, \zeta \xi \colon as \to bt)$ is an object of $a/b$ and  $[\id,\id]_i$ is an element of the source coend whose image is $[\zeta\xi, \id]_{as}=[\zeta,\xi]_g$.
\end{proof}

\begin{Lem} \label{Lem:R-mates}
Consider the pseudo-functors $\cat R_!$ and $\cat R^*$ of \Cref{Lem:R-pseudo-funs}. Their 2-cell images, as well as their structural isomorphisms, are mates under the adjunctions of \Cref{Lem:R-adj} (after inverting).
\end{Lem}

\begin{proof}
Once again, a direct inspection of all definitions yields the result. 
Explicitly, for every 2-cell $\alpha\colon u\Rightarrow v\colon H\to G$ in $\gpd$ we must verify that the left-hand side composite
natural transformation $\cat R^*(u)\Rightarrow \cat R^*(v)$
\[
\xymatrix@C=8pt@R=16pt@L=6pt{
\cat R^*(u) \ar@{=>}[d]_{\simeq} &&
 G(g,uh) \ar[d] &&
  \xi \ar@{|->}[d] \\
\Id_H \; \cat R^*(u) \ar@{=>}[d]_{\eta \, \circ \, \id} &&
 H(h,h) \times G(g,uh) \ar[d] &&
  [\id,\xi] \ar@{|->}[d] \\
\cat R^*(v) \; \cat R_!(v) \; \cat R^*(u) \ar@{=>}[d]_{\id \, \circ \, \cat R_!(\alpha) \, \circ \, \id}  &&
 G(vh,vh) \times G(vh,vh) \times G(g,uh) \ar[d]_{- \circ \alpha_h} &&
  [\id, \id, \xi] \ar@{|->}[d] \\
\cat R^*(v) \; \cat R_!(u) \; \cat R^*(u) \ar@{=>}[d]_{\id \, \circ \, \varepsilon} &&
 G(vh,vh) \times G (uh, vh) \times G(g,uh) \ar[d] &&
  [\id, \alpha_h, \xi] \ar@{|->}[d] \\
\cat R^*(v) \; \Id_G \ar@{=>}[d]_{\simeq} &&
 G(vh,vh) \times G(g,vh) \ar[d] &&
  [\id, \alpha_h\xi] \ar@{|->}[d] \\
\cat R^*(v) &&
 G(g,vh) &&
  \alpha_h\xi 
} 
\]
is equal to $\cat R^*(\alpha)$. 
For every object $(g,h)\in G^\op\times H$, we can follow the fate of (a representative of) an element $\xi\in G(g,uh)$ as on the right-hand side, and the resulting map $\xi \mapsto \alpha_h \xi$ is indeed the component of $\cat F^*(\alpha)$ at~$(g,h)$, as defined. 

Moreover, for any composable $K \overset{v}{\to} H \overset{u}{\to} G$ we must verify that the following composite $\cat R^*(uv)\Rightarrow \cat R^*(v)\circ \cat R^*(u)$ is the inverse~$\fun_{\cat R^*}^{-1}$ of the structure isomorphism of the pseudo-functor~$\cat R^*$:
\[
\xymatrix@C=8pt@R=16pt@L=6pt{
\cat R^*(uv) \ar@{=>}[d]_{\simeq} &&
  \xi \ar@{|->}[d] \\
\Id_K \; \cat R^*(uv) \ar@{=>}[d]_{\eta\,\circ\,\id} &&
 [\id, \xi] \ar@{|->}[d]  \\
\cat R^*(v) \; \cat R_!(v) \; \cat R^*(uv) \ar@{=>}[d]_\simeq &&
   [\id,\id,\xi] \ar@{|->}[d]  \\
\cat R^*(v) \; \Id_H \; \cat R_!(v) \; \cat R^*(uv) \ar@{=>}[d]_{\id \,\circ\, \eta \,\circ\, \id \,\circ\, \id } &&
 [\id,\id,\id,\xi] \ar@{|->}[d]  \\
\cat R^*(v) \; \cat R^*(u) \; \cat R_!(u) \; \cat R_!(v) \; \cat R^*(uv) \ar@{=>}[d]_{\id \,\circ\, \id \,\circ\, \fun_{\cat R_!} \,\circ\, \id } &&
  [\id,\id,\id,\id,\xi] \ar@{|->}[d]   \\
\cat R^*(v) \; \cat R^*(u) \; \cat R_!(u v) \; \cat R^*(uv) \ar@{=>}[d]_{\id \,\circ \, \varepsilon} &&
   [\id,\id,\id \circ \id,\xi] \ar@{|->}[d]  \\
\cat R^*(v) \; \cat R^*(u) \; \Id_G \ar@{=>}[d]_{\simeq} &&
   [\id,\id, \id \circ \xi] \ar@{|->}[d] \\
\cat R^*(v) \; \cat R^*(u) &&
  [\id \circ \id, \xi]  \\
} 
\]
At any $(g,k)\in G^\op\times K$, this amounts to inserting a number of identity maps and composing twice, as indicated in the right-hand colunn, and the resulting map $\xi \mapsto [\id, \xi]$ is indeed the inverse of $\fun_{\cat R_!}$, as we have seen.  

A similar verification, amounting to the counit $\varepsilon\colon \cat R_!(\Id_G) \circ \cat R^*(\Id_G)\Rightarrow \Id_G$ and the left unitor in~$\Biset$ agreeing, shows that the unitors of $\cat R_!$ and $\cat R^*$ are mates.
\end{proof}

\begin{proof}[Proof of \Cref{Thm:main1-intro}]
Apply \Cref{Thm:UP-Span} to the bicategory of bisets, $\cat B:=\Biset$, the pseudo-functors $\cat F_!:=\cat R_!$ and $\cat F^*:=\cat R^*$ of \Cref{Lem:R-pseudo-funs}, and the adjunctions of \Cref{Lem:R-adj}.
The hypotheses~(a), (b) and (c) of the theorem are satisfied by definition, by \Cref{Lem:R-BC}, and by~\Cref{Lem:R-mates} respectively.

It remains to prove the `moreover' part. Let $U\colon H^\op\times G\to \set$ be any biset. 
Then we can construct a span $\cat S(U) = (H  \overset{\;q}{\gets} S(U)  \overset{p}{\to} G) $ and an isomorphism $\cat R\cat S(U)\simeq U$ of bisets, as follows (\cf \eg \cite[\S\,6.4]{Benabou00pp}). 
The groupoid $S(U)$ has object-set $\Obj S(U)=\coprod_{(h,g)\in H^\op\times G} U(h,g)$, and a morphism from $x \in U(h,g)$ to $x'\in U(h',g')$ is a pair $(\beta,\alpha)\in H(h,h')\times G(g,g')$ such that $U(\id,\alpha)(x)=U(\beta,\id)(x')$, with composition induced from $H$ and~$G$. 
The functors $q\colon S(U)\to H$ and $p\colon S(U)\to G$ map an object $x \in U(h,g)\subseteq \Obj S(U)$ to its `source'~$h$ and `target'~$g$, respectively, and a morphism $(\beta, \alpha)$ to $\beta$ and~$\alpha$.
The component at $(h,g)\in H^\op\times G$ of the natural isomorphism $\cat R\cat S (U)\overset{\sim}{\Rightarrow}U$ is the evaluation map 
\[
\big( \cat R_! (p)\otimes_{S(U)} \cat R^*(q) \big)(h,g) = \int^{x \in S(U)} G(px,g) \times H(h,qx) \overset{\sim}{\longrightarrow}  U(h,g)
\]
sending $[\alpha,\beta]_x \mapsto U(\beta, \alpha)(x)$, which is easily seen to be bijective and natural.
\end{proof}

\section{Application: biset functors vs global Mackey functors}
\label{sec:bisetfun}%

In this section we derive \Cref{Cor:bisetfun} from our previous results.
There is not much left for us to do, in fact, besides recalling a few more details and putting everything together.
As before, fix a commutative ring~$\kk$ of coefficients.

\begin{Ter}
\label{Rec:pih}
Let $\cat B$ be any bicategory. Its \emph{1-truncation} or \emph{classifying category} $\pih \cat B$ is the (ordinary) category with the same objects as $\cat B$ and with morphisms the isomorphism classes of 1-morphisms of~$\cat B$. Any pseudo-functor $\cat F\colon \cat B\to \cat B'$ induces a functor $\pih \cat F\colon \pih \cat B\to \pih \cat B'$ in the evident way, by sending a class $[f]$ to $[\cat Ff]$.
\end{Ter}

\begin{Ter}
\label{Rec:semi-add}
A category is \emph{semi-additive} if it is enriched over commutative monoids (\ie every Hom set is equipped with an associative unital sum operation for which composition is bilinear) and if it admits arbitrary finite direct sums (a.k.a.\ biproducts) $X_1\oplus \cdots \oplus X_n$ of its objects, including a zero object (empty direct sum)~$0$. 
If $\cat C$ is any semi-additive category, we may construct a $\kk$-linear additive category~$\kk \cat C$, its \emph{$\kk$-linearization}, with the same objects and with Hom $\kk$-modules given by first group-completing the monoid and then extending scalars:  $\kk \cat C(X,Y):= \kk \otimes_\mathbb Z K_0(\cat C(X,Y))$. 
There is an evident functor $\cat C\to \kk\cat C$ which is initial among functors to $\kk$-linear additive categories. See \eg \cite[App.\,A.6]{BalmerDellAmbrogio20}.
\end{Ter}

\begin{Exa}
\label{Exa:Sp(G)}
Similarly to~$\Span$, one may consider the bicategory $\cat B=\Span(G\sset)$ of finite $G$-sets, spans of maps in~$G\sset$, composed by taking pull-backs, and morphisms of spans. 
Its 1-truncation $\pih\Span(G\sset)$ is a well-known semi-additive category, which already appeared as ``$\spancat(G)$'' in the proof of \Cref{Lem:Yo}; the $\kk$-linearization of the latter is the category $\spancat_\kk(G)$ of~\Cref{Rec:span}. 
Next, we apply the same constructions to spans and bisets of groupoids. 
\end{Exa}

\begin{Lem}
\label{Lem:semi-add}
Consider the bicategories $\Span$ and $\Biset$ of Constructions \ref{Cons:Span} and~\ref{Cons:Biset}.
The disjoint sums of groupoids induce on their 1-truncations $\pih \Span$ and $\pih \Biset$ the structure of semi-additive categories (\Cref{Rec:semi-add}).
\end{Lem}

\begin{proof} 
The sum of any two (parallel) spans $H \overset{\;b}{\gets} S \overset{a}{\to} G$ and $H \overset{\;\;b'}{\gets} S' \overset{a'}{\to} G$ is given by the span $H \overset{(b,b')}{\longleftarrow} S \sqcup S' \overset{(a,a')}{\longrightarrow} G$, and the zero span by $H \gets \emptyset \to G$. For bisets, the sum of $U,V\colon H^\op\times G\to \set$ is the object-wise coproduct $U\sqcup V$, and the zero biset is the constant functor $\emptyset\colon H^\op\times G\to \set$.
The zero object is given in both cases by the empty groupoid, $0= \emptyset$.
The direct sum of two groupoids $G_1,G_2$ is given in $\Span$ and~$\Biset$, respectively (and with notations as in \Cref{sec:R}) by
\begin{equation} \label{eq:dir-sums}
\xymatrix{
G_1 \ar@<.5ex>[r]^-{(i_1)_!} & G_1\sqcup G_2 \ar@<-.5ex>[r]_-{(i_2)^*} \ar@<.5ex>[l]^-{(i_1)^*} & G_2 \ar@<-.5ex>[l]_-{(i_2)_!}
}
\quad \textrm{ and } \quad
\xymatrix{
G_1 \ar@<.5ex>[r]^-{\cat R_!(i_1)} & G_1\sqcup G_2 \ar@<-.5ex>[r]_-{\cat R^*(i_2)} \ar@<.5ex>[l]^-{\cat R^*(i_1)} & G_2 \ar@<-.5ex>[l]_-{\cat R_!(i_2)}
}
\end{equation}
\ie by the canonical covariant and contravariant images of the two inclusions $i_1\colon G_1 \to G_1\sqcup G_2 \gets G_2\,:\!i_2$ in~$\gpd$.
All verifications are straightforward.

(In fact, even before 1-truncating, \eqref{eq:dir-sums} are direct sums in the bicategorical sense; and the sum of 1-morphisms is actually a categorical direct sum, so that the Hom categories are themselves semi-additive; \cf \cite[App.\,A.7]{BalmerDellAmbrogio20} and \cite[Prop.\,3.15]{DellAmbrogio19pp}.) See \cite[\S\,4.3]{Huglo19pp} for more details.
\end{proof}

\begin{Not}
\label{Not:spank-bisk}
As in the introduction, we write 
\[
\spank := \kk \tau_1 (\Span  )
\quad\textrm{ and } \quad
\bisetcat_\kk := \kk \tau_1 ( \Biset )
\]
for the $\kk$-linearization (\Cref{Rec:semi-add}) of the 1-truncation (\Cref{Rec:pih}) of the bicategories of spans and bisets. 
The former makes sense by \Cref{Lem:semi-add}. 
Then 
\[
\cat M := \Rep \spank 
\quad\textrm{ and } \quad
\cat F := \Rep \bisk
\]
are, respectively, the category of \emph{global Mackey functors} and of \emph{biset functors}.
\end{Not}

\begin{Lem} \label{Lem:rigid-for-R}
Both $\spank$ and $\bisk$ are rigid $\kk$-linear tensor categories, with tensor products of objects and maps induced by the Cartesian product of groupoids.
In particular, we may equip $\cat M$ and $\cat F$ with the associated Day convolutions. 
\end{Lem}

\begin{proof}
The tensor product of two spans $H \overset{\;b}{\gets} S \overset{a}{\to} G$ and $H' \overset{\;\;b'}{\gets} S' \overset{a'}{\to} G'$ is
\[
\xymatrix{
H \times H' & \ar[l]_-{b\times b'} S \times S' \ar[r]^-{a \times a'} & G \times G'
}
\]
and the tensor product of two bisets $U\colon H^\op\times G\to \set$ and $U'\colon H'^\op\times G'\to \set$ is 
\[
\xymatrix{
(H \times H')^\op\times (G\times G')
 \simeq 
  (H^\op \times G ) \times (H'^\op \times G')   
     \ar[r]^-{U\times U'} & \set 
}.
\]
In both cases the unit object is the trivial group~$1$. The rest is similarly straightforward. Again, consult \cite[\S\,4.3]{Huglo19pp} for details if necessary.\footnote{And again, both rigid tensor structures should be mere shadows of rigid tensor structures on the \emph{bicategories} $\Span$ and $\Biset$, in a suitable sense, but we have not pursued this.}
\end{proof}

\begin{Rem} 
\label{Rem:bisetfun-comp}
The usual definition of the category of biset functors does not mention groupoids, only groups; \cf \cite{Bouc10}. 
More precisely, \emph{loc.\,cit.\ }defines $\cat F:= \Rep \bisk(\group)$, where $\bisk(\group)\subset \bisk$ is the full subcategory whose objects are groups.
However, the latter inclusion functor is easily seen to be the \emph{additive hull} of~$\bisk(\group)$ and therefore it induces an equivalence $\Rep \bisk \overset{\sim}{\to} \Rep \bisk(\group) $ of functor categories, whence the agreement of the two definitions of biset functors (\cf \cite[Rem.\,6.5]{DellAmbrogio19pp}).
The Day convolution of biset functors and of global Mackey functors are studied, respectively, in \cite[Ch.\,8]{Bouc10} and~\cite{Nakaoka16a}.
\end{Rem}

\begin{Lem}
\label{Lem:pihR}
The pseudo-functor $\cat R\colon \Span \to \Biset$  of \Cref{Thm:main1-intro} induces a well-defined, essentially surjective, full $\kk$-linear tensor  functor $\kk \pih \cat R\colon \spank \to \bisk$.
\end{Lem}

\begin{proof}
It is immediate from \eqref{eq:dir-sums} that the induced functor $\pih \cat R\colon \pih \Span \to \pih \Biset$ is additive, \ie preserves direct sums of objects and therefore also the addition of maps (\cf \cite[Rem.\,A.6.7]{BalmerDellAmbrogio20} if necessary). 
In particular, it extends uniquely to a $\kk$-linear functor $\kk \pih \cat R$ between the $\kk$-linearizations.
This is obviously surjective on objects, and it is full by the `moreover' part of \Cref{Thm:main1-intro}.

It remains to see that $\kk \pih \cat R$ is symmetric monoidal. 
Indeed it is strictly so through the identity maps 
$\unit = 1 = \cat R(\unit)$
and
$\cat R(G) \otimes \cat R(G') = G \times G' = \cat R(G\otimes G') $,
because there are easily-guessed isomorphisms of bisets  (\cf \cite[Lem.\,4.3.11]{Huglo19pp})
\[
\cat R(a_!b^*) \otimes \cat R(a'_! b'^*) \overset{\sim}{\longrightarrow} \cat R(a_! b^* \otimes a'_!b'^*)
\]
showing that the functors $\otimes \circ (\kk \pih \cat R \times \kk \pih \cat R)$ and $\kk \pih  \cat R \circ \otimes$ are equal.
\end{proof}

\begin{proof}[{Proof of \Cref{Cor:bisetfun}}]
By Lemmas~\ref{Lem:rigid-for-R} and \ref{Lem:pihR}, the categories $\cat C:= \spank$ and $\cat D:= \bisk$ and the functor $F:= \kk \pih \cat R \colon \cat C\to \cat D$ satisfy all the hypotheses of \Cref{Thm:main2-intro}. 
This proves most of the claims of the corollary. 

It remains to show that a global Mackey functor is (isomorphic to the restriction of) a biset functor if and only if it satisfies the deflative relation,
$
\mathrm{def}^G_{G/N} \circ \mathrm{inf}^G_{G/N}= \id_{M(G/N)}
$,
whenever $N$ is a normal subgroup of a group~$G$. 
Here, for the sake of familiarity, we have used the classical notations
$\mathrm{def}^G_{G/N}= M([G = G \to G/N])$ 
and 
$\mathrm{inf}^G_{G/N}= M([G/N \gets G = G])$
for the \emph{deflation} and \emph{inflation} maps of a functor $M\in \cat M$, where $G\to G/N$ is the quotient map.

Equivalently, we must show that the kernel of the realization functor $F=\kk \pih \cat R$ is generated, as a $\kk$-linear categorical ideal of~$\spank$, by the corresponding differences of spans, \ie (after computing the obvious iso-comma square up to equivalence) by
\[
[G/N \gets G \to G/N] - [G/N = G/N = G/N]
\quad \textrm{ for all } N \unlhd G \,.
\]
While it is easy to see that these elements belong to the kernel (just compute $\cat R([G/N \gets G \to G/N])\simeq G/N(-,-)$), it is \emph{a~priori} not obvious to show that they generate it. 
This can be achieved by comparing two explicit presentations of $\spank$ and~$\bisk$, as done in the proof of \cite[Thm.\,6.9]{DellAmbrogio19pp}, to which we refer. Alternatively, one may consult the~-- possibly less transparent but ultimately equivalent~-- calculations in \cite[App.\,A]{Ganter13pp} or \cite[\S\,6]{Nakaoka16}.
\end{proof}

\begin{Rem}
Not every global Mackey functor satisfies the deflative relations, for instance the tensor unit $\unit = \spank(1,-)$ does not; see \cite[\S\,5.4]{Nakaoka16}. As deflative Mackey functors form a tensor ideal, if the unit \emph{were} deflative so would everyone.
\end{Rem}


\bibliographystyle{alpha}

\end{document}